%% file: main.tex
\documentclass[a4paper,12pt]{article}

\usepackage[utf8]{inputenc}
\usepackage[T1]{fontenc}
\usepackage{float, etoolbox}
\usepackage{mathrsfs}
\usepackage{tikz, pgf}
\usetikzlibrary{arrows,automata,positioning}
\usepackage{graphicx, longtable, wrapfig, rotating}
\usepackage[normalem]{ulem}
\usepackage{amsmath, amsthm}
\usepackage{textcomp, marvosym, wasysym, amssymb, capt-of, hyperref}
\usepackage{comment}
\usepackage{color}
\usepackage{subfigure}
\usepackage{framed, xcolor}
\usepackage[ruled, vlined, linesnumbered]{algorithm2e}

\input{macros}

\begin{document}
\author{Marcelo Forets$^{1}$ \and Amaury Pouly$^{2}$}
\date{\today}
\title{Explicit Error Bounds for Carleman Linearization}
\footnotetext[1]{VERIMAG, Université Grenoble Alpes, France.}
\footnotetext[2]{Max-Planck Institute for Software Systems, Saarbrucken, Germany.}
\maketitle

\begin{abstract}
We revisit the method of Carleman linearization for systems of ordinary differential equations with polynomial right-hand sides. This transformation provides an approximate linearization in a higher-dimensional space through the exact embedding of polynomial nonlinearities into an infinite-dimensional linear system, which is then truncated to obtain a finite-dimensional representation with an additive error. To the best of our knowledge, no explicit calculation of the error bound has been studied. In this paper, we propose two strategies to obtain a time-dependent function that locally bounds the truncation error. In the first approach, we proceed by iterative backwards-integration of the truncated system. However, the resulting error bound requires an a priori estimate of the norm of the exact solution for the given time horizon. To overcome this difficulty, we construct a combinatorial approach and solve it using generating functions, obtaining a local error bound that can be computed effectively.
\end{abstract}

\paragraph{Keywords:} 

carleman linearization, polynomial ODEs, infinite-dimensional systems, guaranteed integration, nonlinear control theory.

\section{Introduction}
\label{sec:Introduction}

\input{introduction}

\section{Preliminaries}
\label{sec:Preliminaries}

\input{preliminaries}

\section{Carleman embedding}
\label{sec:CarlemanEmbedding}

\input{carleman_embedding}

\section{Main results}
\label{sec:MainResults}

\input{main_results}

\section{Proofs}
\label{sec:Proofs}

\input{proofs}

\section{Conclusion}
\label{sec:Conclusion}

\input{conclusion}

\subsection*{Acknowledgements} M.F. acknowledges stimulating discussions with Goran Frehse, Thao Dang and Victor Magron at the beginning stages of this work. We are indebted to Pablo Rotondo for help in Lemma \ref{lemma:RecurrentIntegralTN}, to Iosif Pinelis for advice on generating function inequalities, and to Marko Riedel for valuable insight into Egorychev's method.
\bibliographystyle{abbrv} 
\bibliography{carlin}

\clearpage 

\appendix

\section{Proof of Lemma \ref{lemma:RecurrentIntegralTN}}
\label{app:ProofPropIntegrals}

\input{proofpropintegrals}

\section{Proof of Lemma \ref{lemma:recurrenceK2}}
\label{app:ProofRecurrencePowerSeries}

\input{proofrecurrencepowerseries}

\end{document}

%% file: macros.tex
\newcommand{\R}{\mathbb{R}}
\newcommand{\N}{\mathbb{N}}

\newcommand{\joli}[1]{\mathcal{#1}}
\newcommand{\E}{\joli{E}}
\newcommand{\A}{\joli{A}}

\newcommand{\norm}[1]{\Vert #1 \Vert}
\newcommand{\norminf}[1]{\Vert #1 \Vert_\infty}
\newcommand{\transp}{\mathrm{T}}

\newcommand{\Idn}{\ensuremath{\mathbb{I}_{n\times n}\xspace}}
\newcommand{\Idm}{\ensuremath{\mathbb{I}_{m\times m}\xspace}}

\newcommand{\ii}{\mathrm{i}}
\newcommand{\dd}{\mathrm{d}}

\newcommand\encircle[1]{%
	\tikz[baseline=(X.base)] 
	\node (X) [draw, shape=circle, inner sep=0] {\strut #1};}

\tikzset{fontscale/.style = {font=\relsize{#1}}}

\theoremstyle{plain}
\newtheorem{theorem}{Theorem}[section]
\newtheorem{lemma}[theorem]{Lemma}
\newtheorem{proposition}[theorem]{Proposition}

\theoremstyle{definition}
\newtheorem{definition}[theorem]{Definition}

\newtheorem{example}{Example}
\newtheorem{remark}{Remark}
\if{
	\newcommand{\qed}{\nobreak \ifvmode \relax \else
		\ifdim\lastskip<1.5em \hskip-\lastskip
		\hskip1.5em plus0em minus0.5em \fi \nobreak
		\vrule height0.75em width0.5em depth0.25em\fi}
}\fi

\newcounter{todocounter}
\colorlet{shadecolor}{ gray!15 }

%% file: introduction.tex
In 1932, Carleman devised a method \cite{carleman1932application} (now known as Carleman linearization)
to embed a nonlinear system of differential equations $x'(t)=f(x(t))+u(t)g(x(t))$ of finite dimension
into a system of bilinear differential equations 
\begin{equation}\label{eq:infdimSystem}
y'(t)=\mathcal{A}y(t)+u(t)\mathcal{B}y(t)
\end{equation} 
of \emph{infinite dimension}. 
By truncating the obtained bilinear system at finite orders, one obtains a systematic
way of creating arbitrary-order approximation of the solutions of the nonlinear system.
In particular, when put together with results on the Volterra series of bilinear systems,
it provides an effective way of computing the Volterra series for a large class
of nonlinear systems where $f$ and $g$ are analytic. This approach was initiated in \cite{Brockett1976}
and refined in a number of papers in the years that followed. It has been particularly
successful for proving results about the observability \cite{mozyrska2008carleman, sira1988algebraic}, stability \cite{LoparoB1978} and controllability \cite{mozyrska2006dualities} of nonlinear
systems.  More recently these methods have been applied in stochastic state estimation and controller design \cite{germani2005filtering, rauh2009carleman} and model order reduction \cite{bai2002krylov}. See \cite{Brockett2014} for a survey on the subject. In the particular case where the input $u$ is a nonlinear feedback, the above
approach can be refined to obtain more explicit convergence and stability conditions
such as in \cite{LoparoB1978}.

Linear systems of infinite-dimensional ODEs were first studied in the late 1940's and early 1950's \cite{arley1944theory, bellman1947boundedness}. Those initial works focused on the study of existence and uniqueness of solutions. In later works, Chew, Shivakumar and Williams \cite{chew1980error}, and Shivakumar \cite{shivakumar1988iterative} provided error formulas for the difference between the solution of the exact ODE and the solution obtained by finite order truncations. Well-posedness of the Cauchy problem of infinite-dimensional ODEs was studied by Borok \cite{Bor82}. Through an extension based on using the logarithmic norm, Marinov \cite{marinov1986truncation} obtained a time-dependent error bound which converges to zero for certain classes of systems. These bounds apply to the general case of bounded operators, e.g.~when $\mathcal{A}$ is a bounded operator in the sequence space $\ell_1$. We refer to the monograph \cite{shivakumar2016infinite} for applications of that theory. However, the matrix operator \eqref{eq:infdimSystem} obtained though Carleman linearization is in general unbounded, and these results do not apply.

An important subclass of nonlinear systems are polynomial differential equations.
Indeed, many systems can be rewritten as polynomial vector fields by introducing
more variables and, in fact, any polynomial system can be reduced to a second-order
polynomial one \cite{CPSW05, hernandez1998algebraic}. Convergence of the infinite-dimensional exponential map associated to a polynomial vector field was studied by Winkel \cite{winkel1997exponential}. An error estimate is given, but it is coarse, i.e.~it is time-independent. The Volterra series approach yields some interesting bounds for polynomial systems \cite{Bullo200}. But as far as we
are aware, previous results on the subject are only concerned with proving that the approach is
sound \cite{krener1974linearization}, i.e. the truncated linearization converges  as the order increases,
and not in obtaining explicit bounds. In particular, a number of
bounds based on operator norms cannot easily be evaluated on a computer except for special classes or systems.

In this paper, we consider the nonlinear ordinary differential equation
\begin{equation}
	x'(t)= f(x(t)),
\end{equation}
where $f \in \R^n[x]$ is a polynomial vector field with domain $\R^n$, and $f(0)=0$. We revisit
the Carleman linearization method for such systems and give an error bound based for the truncated linear system based on backwards-integrating Volterra
series, a generalization of the technique in \cite{bellman1962some}. We then develop an alternative approach by studying the power series of the solution, and obtain an explicit formula for the error bound using generating functions. The obtained bound can be exponentially better than the first one as the order
increases. Moreover, these bounds have a very simple and explicit expression that can be used in a numerical algorithm.

We can summarize the contributions of this article as follows:

\begin{itemize}
	\item We provide an upper bound for the truncation error of the Carleman linearization method. This bound is explicit in the system's coefficients and initial state, but depends on an a priori estimate of the norm of the solution. It is obtained by estimating the solution of the truncated Volterra series.
	
	\item Then, we provide a second upper bound which is explicit, i.e.~it does not depend on any a priori estimate of the exact solution. We obtain this bound constructing the series for the error term and solving the general recurrence using generating functions.
\end{itemize}

In  both cases, our formulas apply to any polynomial ODE with the origin as a stationary point. The results are obtained through an adequate reformulation into a quadratic polynomial ODE. We implemented the required transformations, and the explicit error formulas, in the computer algebra system SageMath \cite{sagemath}.

We begin in Section \ref{sec:Preliminaries} with preliminaries: Kronecker product, Kronecker power and the logarithmic norm. We continue with the formalism for the Carleman embedding method \ref{sec:CarlemanEmbedding}. We formalize our explicit error bounds in Section \ref{sec:MainResults}, and provide the proofs in Section \ref{sec:Proofs}. We conclude and comment on ideas for future work in Section \ref{sec:Conclusion}.

%% file: preliminaries.tex
In this section we setup the notations used in this paper, review the Kronecker product for vectors and matrices, and recall the definition and basic properties of the logarithmic norm of a matrix.

\subsection{Notation}
\label{ssec:Notation}

Let $\N = \{ 1,2,\ldots \}$ be the set of positive integers and $\R$ the set of real numbers. Let $\Idn$ denote the identity matrix of order $n$. Hereafter, $\norm{\cdot}$ denotes the supremum norm in the Euclidean space $\R^n$, 
\begin{equation*}
\norm{x} := \norminf{x} = \max_{i} |x_i|,\qquad x \in \R^n. \label{def:xNormInf}
\end{equation*}
For the norm of a matrix $A = (a_{ij}) \in \R^{n\times n}$ we refer to the induced norm, namely 
\begin{equation}
\norm{A} := \max_{\Vert x \Vert = 1} \Vert A x \Vert = \max_{i} \sum_{j} |a_{ij}| \label{prop:SupNormInduced},
\end{equation}
which is the maximum absolute row sum of the matrix.

\subsection{Kronecker product and Kronecker power}
\label{ssec:KroneckerProduct}


For any pair of vectors $x \in \R^n$, $y\in \R^m$, their \textit{Kronecker product}\footnote{For further properties on the Kronecker product than those recalled here, we refer to \cite{zhang2011matrix} or \cite{steeb2011matrix}.} $w \in \R^{mn}$ is
$$
w = x\otimes y = (x_1 y_1, x_1 y_2,\ldots, x_1 y_m, x_2 y_1,\ldots,x_2 y_m, \ldots, x_n y_1,\ldots, x_n y_m)^\transp.
$$
This product is not commutative. For matrices the definition is analogous: if $A \in \R^{m\times n}$ and $B \in \R^{p \times q}$, then $C \in \R^{mp \times nq}$ is
\begin{equation*}
C = A \otimes B = \begin{pmatrix}
a_{11}B &\ldots&a_{1n}B \\
\vdots & & \vdots \\
a_{m1}B & \ldots &a_{mn}B 
\end{pmatrix}.
\end{equation*} 
We recall next that the supremum norm satisfies the \emph{crossnorm property} \cite{lancaster1972norms}.
 
\begin{lemma} \label{lemma:crossnormSupMatrixNorm}
Let $A = (a_{ij}) \in \R^{m\times n}$, $B  = (b_{ij}) \in \R^{p\times q}$. Then $\Vert A\otimes B\Vert = \Vert A\Vert~\Vert B \Vert.$

\begin{proof}
The matrix elements of $A\otimes B$ are 
\begin{equation*}
(A\otimes B)_{p(r-1)+v, q(s-1)+w} = a_{rs} b_{vw}, 
\end{equation*}
for all $1\leq r\leq m$, $1\leq s \leq n$, and $1\leq v\leq p$, $1\leq w \leq q$. From \eqref{prop:SupNormInduced}, 
\begin{subequations}
\begin{align*}
\norm{A\otimes B} &= \max_{1\leq r \leq m} \max_{1\leq v \leq p} \sum_{s=1}^n\sum_{w=1}^q |(A\otimes B)_{p(r-1)+v, q(s-1)+w}| \\
&= \max_{1\leq r \leq m} \max_{1\leq v \leq p} \sum_{s=1}^n\sum_{w=1}^q |a_{rs}|~|b_{vw}| \\
&= \max_{1\leq r \leq m} \max_{1\leq v \leq p} \sum_{s=1}^n |a_{rs}| ~ \sum_{w=1}^q |b_{vw}| \\
&= \max_{1\leq r \leq m} \sum_{s=1}^n |a_{rs}|  ~ \max_{1\leq v \leq p}  \sum_{w=1}^q |b_{vw}|  \\
&= \norm{A}~\norm{B}.
\end{align*} 
\end{subequations} 

\end{proof}

\end{lemma}

The \textit{Kronecker power} is a convenient notation to express all possible products of elements of a vector up to a given order, and it is denoted
\begin{equation}
x^{[i]} :=  \underset{\text{i times}}{\underbrace{x\otimes \cdots \otimes x}},\qquad x \in \R^n.\label{eq:xKroneckerpoweri0}
\end{equation} 
Moreover, $\dim x^{[i]} = n^i$, and each component  of $x^{[i]}$ is of the form $x_1^{\omega_1} x_2^{\omega_2} \cdots x_n^{\omega_n}$ for some multi-index $\omega \in \N^n$ of weight $|\omega| = i$. It follows from Lemma \ref{lemma:crossnormSupMatrixNorm} that $\norm{x\otimes y} = \norm{x}~\norm{y}$ for any $x\in \R^n$ and $y \in \R^m$, and by extension the supremum norm is homogeneous with respect to the Kronecker power,
\begin{equation}
\norm{x^{[i]}} = \norm{x}^i, \qquad i\in \N. \label{eq:ithpowerIdentity}
\end{equation}

\begin{example}
	For $n=2$, the Kronecker powers up to order three are: $x^{[1]}= x$, $x^{[2]} = (x_1^2,x_1 x_2,x_2 x_1,x_2^2)^\transp$ and  $x^{[3]} = (x_1^3,x_1^2 x_2, x_1^2 x_2, x_1 x_2^2,    x_2 x_1^2,x_2^2 x_1 , x_1 x_2^2,x_2^3 )^\transp$.
\end{example}

\subsection{Logarithmic norm} 

The logarithmic norm of $A \in \mathbb{C}^{n\times n}$ with respect to a given matrix norm (induced by some vector norm), is defined as a right G\^{a}teaux derivative \cite{desoer2009feedback}, namely the $h \to 0^+$ limit of $(\Vert \Idn + h A\Vert - \norm{A})/h$. For the supremum norm from \eqref{prop:SupNormInduced} we can deduce that
\begin{equation*}
\mu(A) := \mu_{\infty}(A)=\max_{i} \text{Re } a_{ii} + \sum_{j\neq i} \vert a_{ij} \vert.
\end{equation*}
The logarithmic norm satisfies the following properties \cite{soderlind2006logarithmic}: (i) it is sub-additive, i.e.~$\mu(A+B) \leq \mu(A) + \mu(B)$; (ii) it is upper bounded by the norm of $A$, i.e.~$\mu(A) \leq \norm{A}$; and (iii) $\norm{e^A} \leq e^{\mu(A)}$.

We remark that there is a slight abuse of notation, since the logarithmic norm is not a norm in the usual sense (it can take negative values).

\begin{example}
Let 
$$
A = B = \begin{pmatrix}
0 & 1 \\ -1 & -2
\end{pmatrix}.
$$
Then $\norm{A} = \norm{B} = 3$, and $\norm{AB}= 5$, but $\mu(A) = \mu(B) = 1$,  and $\mu(AB) = 5$. Their Kronecker product is
$$
A\otimes B = 
\begin{pmatrix}
	0 & 0 & 0 & 1 \\
	0 & 0 & -1 & -2 \\
	0 & -1 & 0 & -2 \\
	1 & 2 & 2 & 4
\end{pmatrix},
$$
and we see that $\norm{A\otimes B} = 9$ as well as $\mu(A\otimes B) = 9$. 
\end{example}

The previous example shows that for the logarithmic norm, neither sub-multiplicativity nor the crossnorm property hold in general. However, the following particular case is sufficient for our purposes. 

\begin{lemma} \label{lemma:crossnormLogNorm}
The logarithmic norm satisfies $\mu(A \otimes \Idm) = \mu(A)$ for any $A \in \mathbb{R}^{n\times n}$.
\begin{proof}
Similarly as in Lemma \ref{lemma:crossnormSupMatrixNorm}, we use that 
$$
(A \otimes \Idm)_{m(r-1)+v, m(s-1)+w} = a_{rs} \delta_{vw}
$$ 
for all $1\leq r \leq n$, $1 \leq s \leq n$, $1 \leq v \leq m$, $1 \leq w \leq m$. Moreover, for real $A$,
\begin{align*}
\mu(A\otimes \Idm) &= \max_{1\leq r \leq n} \max_{1\leq v \leq m} \sum_{s=1}^n \sum_{w=1}^m |a_{rs}| |\delta_{vw}| + a_{rr} - |a_{rr}| \\
&= \max_{1\leq r \leq n} \max_{1\leq v \leq m} \sum_{s=1}^n |a_{rs}| + a_{rr} - |a_{rr}| \\
&= \max_{1\leq r \leq n}  \sum_{s=1}^n |a_{rs}| + a_{rr} - |a_{rr}| \\
&= \mu(A).
\end{align*}

\end{proof}
	
\end{lemma}

%% file: carleman_embedding.tex
This section is devoted to reviewing the Carleman linearization technique. The infinite-dimensional realization and the finite-dimensional truncation are treated in \ref{ssec:infinidim-review} and \ref{ssec:truncation-review} respectively. We provide in \ref{sec:reductionQuadratic} a reduction algorithm that will be frequently used in the proofs. For additional background on Carleman linearization, we refer to \cite[Chapter 3]{rugh1981nonlinear} or to the comprehensive book \cite{kowalski1991nonlinear}.

\subsection{Infinite-dimensional ODE} \label{ssec:infinidim-review}

Consider the initial-value problem  (IVP)
\begin{equation}
\left\{
\begin{aligned}
x'(t) &= F_1 x + F_2 x^{[2]} + \ldots +  F_k x^{[k]}, \\
 x(0) &= x_0 \in \R^n,\qquad t \in I.
\end{aligned}
\right.
\label{eq:S_ode}
\end{equation}
We assume that the matrix-valued functions $F_j \in \R^{n \times n^j}$ are independent of $t$. Here $k$ is the degree of the polynomial ODE.

\begin{definition}
The \textit{transfer matrix} $A^i_{i+j-1} \in \R^{n^i \times n^{i+j-1}}$ for $i\in \N$, $j = 1,\ldots,k$, is
\begin{align}
A^i_{i+j-1} &:= \sum_{\nu=1}^i   \overset{\text{i factors}}{\overbrace{\Idn \otimes \cdots \otimes \underset{\underset{\nu-\text{th position}}{\uparrow}}{F_j} \otimes  \cdots  \otimes \Idn }}. \label{eq:defiBij}
\end{align}	
We represent this linear map with the diagram
\begin{equation*}
\begin{tikzpicture}[->,>=stealth',shorten >=1pt,auto,
semithick]
\tikzstyle{every state}=[fill=white,draw=black,text=black,inner sep=0pt,minimum size=1cm]

\node[state](A)   {$i$};

\node[state] (B) [right of=A,node distance=2.2cm] {$q$};

\path[->] (A) edge[bend left]          node {$A^{i}_{q}$} (B);

\path[-, loosely dotted, line width=2pt] (A) edge[]              node {} (B);
\end{tikzpicture}
\end{equation*}
where $q = i+j-1$. We will extensively use this notation in the proof sections.
\end{definition} 

It is convenient to define the $i$-th \emph{block of auxiliary variables} as
\begin{equation}
y_i := x^{[i]},\qquad i\in \N. \label{eq:changeofvariables}
\end{equation}
Since $y_i$ is a Kronecker power, $\dim y_i  = n^i$.
 
\begin{proposition} \label{prop:xiDynamics}
If $x : I \to \R^n$ solves \eqref{eq:S_ode} in the interval $I = [0,T] \subset \R^+$, then
 \begin{equation*}
  y'_i= \sum_{j=0}^{k-1} A^i_{i+j} y_{i+j},\qquad i \in \N.  \label{eq:XievolutionAbstract}
 \end{equation*}

 \begin{proof}
  Differentiating Eq. \eqref{eq:xKroneckerpoweri0} and applying Leibniz rule, it follows that
  \begin{subequations}  
  \begin{align*}
   y'_i &= x' \otimes x \otimes \cdots \otimes x + \ldots + x \otimes \cdots \otimes x \otimes x' \\
   &= \sum\limits_{\nu=1}^i \left( x \otimes \cdots \otimes \sum\limits_{j=1}^k F_j x^{[j]} \otimes \cdots \otimes x \right).
  \end{align*}
  \end{subequations}
  From linearity of the Kronecker product we can exchange the sums,
\begin{subequations}  
  \begin{align*}
y'_i &= \sum_{j=1}^k \sum_{\nu=1}^i \left( x \otimes \cdots \otimes  F_j x^{[j]} \otimes \cdots \otimes x \right) \\ 
  &= \sum_{j=1}^k \sum_{\nu=1}^i \left(\Idn \otimes \cdots \otimes F_j \otimes \cdots \otimes \Idn \right) ( x\otimes \cdots \otimes x^{[j]} \otimes \cdots \otimes x ) \\
  &= \sum_{j=1}^k A^{i}_{i+j-1} x^{[i+j-1]} \\
  &= \sum_{j = 0}^{k-1} A^{i}_{i+j} y_{i+j}.
  \end{align*}
  \end{subequations}
 \end{proof}
 
\end{proposition}

It is convenient to express Proposition \ref{prop:xiDynamics} in matrix form. This can be achieved concatenating all blocks $y_i$ into an infinite-dimensional vector $y := (y_1,y_2,\ldots)^\transp$,  $y \in \R^\N$. Formally, $y(t)$ satisfies the IVP

\begin{equation}
\left\{
\begin{aligned}
y'(t) &= \A y(t), \\
y_i(0) &= (x(0))^{[i]} = x_0^{[i]},\qquad \forall i \in \N,
\end{aligned}
\right.
\label{eq:InfSysMatrix}
\end{equation}
where $\A$ is the infinite-dimensional block upper-triangular matrix
\begin{equation*}
\A := \begin{pmatrix}
A_1^1 & A_2^1 & A_3^1 & \ldots & A_{k}^1 & 0 & 0 & \ldots \\ 
0& A_2^2 & A_3^2 & \ldots & A_{k}^2 & A_{k+1}^2&0 & \ldots \\
0 & 0 & A_3^3 & \ldots  & A_{k}^3 & A_{k+1}^3  & A_{k+2}^3 & \ldots \\   
\vdots &  \vdots &  \vdots &    & \vdots &\vdots & \vdots &\\
\end{pmatrix}. \label{eq:Adefi}
\end{equation*}

This particular structure can be exploited both from a theoretical and from a practical point of view. In particular, we will make use of the following recurrence formula.


\begin{proposition} \label{prop:PropAijFjestimate}
For all $i\geq 1$, $0\leq j \leq k-1$, the estimate $\norm{A^i_{i+j}} \leq i \norm{F_{j+1}}$ holds. 

\begin{proof}

From \eqref{eq:defiBij}, for all $i \geq 2$, $1 \leq j \leq k$, the transfer matrices satisfy 
\begin{equation*}
A^i_{i+j-1} = A^{i-1}_{i+j-2} \otimes \Idn + \Idn^{[i-1]} \otimes  A^1_j.
\end{equation*}
By the triangular inequality,
\begin{equation}
\norm{A^i_{i+j-1}} \leq \norm{A^{i-1}_{i+j-2} \otimes \Idn} + \norm{ \Idn^{[i-1]} \otimes  A^1_j}. \label{eq:InducNorms}
\end{equation}
By the cross-norm property, and since $||\Idn||=1$ for any $n\in \N$, the right-hand side of \eqref{eq:InducNorms} simplifies to
\begin{equation*}
||A^i_{i+j-1}|| \leq ||A^{i-1}_{i+j-2}|| + ||F_j||, \label{eq:InducNorms2}
\end{equation*}
since $A^1_j = F_j $. Applying $i-1$ times the inequality \eqref{eq:InducNorms}, we obtain obtain $||A^i_{i+j-1}|| \leq i ||F_{j}||$, and with the change of variable $j\to j-1$ we obtain the claim.

\end{proof}

\end{proposition}


\subsection{Truncation of the infinite-dimensional system} \label{ssec:truncation-review}

Now we move on to consider the truncated system of order $N$. Here we choose the \emph{null closure conditions} \cite{bellman1962some}, that consists of eliminating the dependence on variables of order exceeding $N$. The $y_i$-variables removed by the truncation start at $N-k+2$, because from Prop. \ref{prop:xiDynamics} we know that each order $i$ is influenced by the variable further at $k-1$ positions at most. Let

\begin{equation*}
\hat{y}'_i := \sum\limits_{j=0}^{k-1} A^i_{i+j} \hat{y}_{i+j}, \qquad i = 1,\ldots,N-k+1
\end{equation*}
and 
\begin{equation*}
\hat{y}'_i := \sum\limits_{j=0}^{\ell-1}  A^i_{i+j} \hat{y}_{i+j}, \qquad i = N-k+2,\ldots,N
\end{equation*}
where $\ell=\min\{k,N+1-i\}$. As previously, we can express this system in matrix form, if we consider the (finite-dimensional) vector $\hat{y} = (\hat{y}_1, \hat{y}_2, \ldots, \hat{y}_N)^\transp$. Then by construction $\hat{y}$ satisfies the IVP:

\begin{equation}
\left\{
\begin{aligned}
\dfrac{d \hat{y}}{dt} &= \A_N \hat{y}, \\
\hat{y}_i(0) &= x^{[i]}(0) = x_0^{[i]}, \qquad \forall i = 1,\ldots,N
\end{aligned}
\right.
\label{eq:TruncSysMatrix}
\end{equation}
The initial condition is compatible with \eqref{eq:InfSysMatrix}, and $\A_N$ is a finite-dimensional, square, block upper-triangular matrix. Since~$\dim A^i_{i+j-1} = n^i \times n^{i+j-1}$,
the order of $\A_N$ is~$(n^{N+1}-n)/(n-1)$.

\subsection{Reduction to the quadratic case}
\label{sec:reductionQuadratic}

A \emph{quadratic system} is an ODE of the form $x' = F_1 x + F_2 x^{[2]}$.
It is well-known that any higher-order polynomial vector field can be brought into this form by introducing new auxiliary variables. We recall the procedure here for self-containment. 

\begin{proposition}\label{prop:quadratic_reduction}
	Consider the $k$-th order system ($k\geq 2$), 
	\begin{equation}
x' = F_1 x + F_2 x^{[2]} + \ldots + F_{k} x^{[k]}. \label{eq:kthOrderSys}
	\end{equation}
	Introducing the variables $\tilde{x}_i := x^{[i]}$ for $1\leq i \leq k-1$,
	the $k$-th order system \eqref{eq:kthOrderSys} reduces to a quadratic system in $\tilde{x}$, that is,
	\begin{equation}
	\tilde{x}' = \tilde{F}_1 \tilde{x} + \tilde{F}_2 \tilde{x}^{[2]}, \label{eq:2ndOrderSys}
	\end{equation}
	where $\tilde{x} := (\tilde{x}_1, \tilde{x}_2,\ldots, \tilde{x}_{k-1})^\transp$, and the matrices $ \tilde{F}_1$ and $\tilde{F}_2$ are given below.
	Moreover, the supremum norm of the linear and quadratic parts satisfy, respectively,
	\begin{equation*}
			\norm{\tilde{F}_1} \leq \max_{1\leq i \leq k-1} (k-i) \sum_{j=1}^i \norm{F_j} ~~ \text{ and } ~~	\norm{\tilde{F}_2} \leq (k-1) \sum_{j=2}^k \norm{F_j}.
	\end{equation*}
		
	\begin{proof}
		Recall from Prop. \ref{prop:xiDynamics} that $dx^{[i]}/dt = \sum_{j=0}^{k-1} A^i_{i+j} x^{[i+j]}$, $k\geq 2$. Hence, each component of $\tilde{x}$ satisfies the dynamics:
				\begin{equation*}
			\left\{
			\begin{aligned}
			\tilde{x}'_1 &= A^1_1 \tilde{x}_1 + A^1_2~ \tilde{x}_2 + \ldots +   A^1_{k-1} \tilde{x}_{k-1} + A^1_k ~ \tilde{x}_1 \otimes \tilde{x}_{k-1} \\ \\
			\tilde{x}'_2 &= A^2_2 \tilde{x}_2 + A^2_3~ \tilde{x}_3 + \ldots + A^2_{k-1} ~ \tilde{x}_{k-1} + A^2_{k} ~ \tilde{x}_1 \otimes \tilde{x}_{k-1} + A^2_{k+1} ~ \tilde{x}_2 \otimes \tilde{x}_{k-1}\\ 
			&\hspace{0.5em}\vdots \\ 
			\tilde{x}'_{k-1} &= A^{k-1}_{k-1} \tilde{x}_{k-1} + A^{k-1}_{k}~ \tilde{x}_1\otimes \tilde{x}_{k-1} + \ldots + A^{k-1}_{2(k-1)} \tilde{x}_{k-1} \otimes \tilde{x}_{k-1}.
			\end{aligned} \right.
			\end{equation*}
		In consequence, $\tilde{x}' = \tilde{F}_1 \tilde{x} + \tilde{F}_2 \tilde{x}^{[2]}$, 
		with the linear part being
		\begin{equation*}
		\tilde{F}_1 := \begin{pmatrix}
		A^1_1 & A^1_2 & A^1_3 & \ldots & A^1_{k-1} \\ 
		0 & A^2_2 & A^2_3 & \ldots & A^2_{k-1} \\
		0 & 0 & \ddots  & & \vdots \\
		\vdots & \vdots &   & \ddots & \vdots  \\
		0 & 0 & \ldots & 0& A^{k-1}_{k-1} 
		\end{pmatrix},
		\end{equation*}
	and the quadratic part
		\begin{equation*}
		\tilde{F}_2 := 
		\left( \begin{array}{cccccccccccc}
		0 & \cdots & 0 & A^1_k & 0 & \cdots & 0 & 0 & 0 & \cdots~\cdots & 0 & 0 \\
		0 & \cdots & 0 & A^2_k & 0 & \cdots & 0 & A^2_{k+1} & 0 & \cdots~\cdots & 0 &0 \\
		\vdots & & \vdots & \vdots & \vdots &   & \vdots & \vdots & \vdots &  & \vdots &\vdots \\
		0 & \cdots & 0 & A^{k-1}_k & 0 & \cdots & 0 & A^{k-1}_{k+1} & 0 & \cdots~\cdots & 0 & A^{k-1}_{2(k-1)} 
		\end{array} \right).
		\end{equation*}
        On the oher hand, since $||A^i_{i+j}|| \leq i ||F_{j+1}||$ from Prop. \ref{prop:PropAijFjestimate}, we find
		\begin{align}\label{eq:upperbound_F1tilde}
		||\tilde{F}_1|| &= \sup \{ ||A^1_1||+||A^1_2||+\ldots + ||A^1_{k-1}||, \nonumber \\
			&\hspace{4em}||A^2_2||+||A^2_3||+\ldots + ||A^2_{k-1}||, \ldots ||A^{k-1}_{k-1}|| \}, \nonumber \\
		&\leq \max \{ ||F_1||+||F_2||+\ldots + ||F_{k-1}||, \nonumber \\
			&\hspace{4em}2(||F_1||+\ldots + ||F_{k-2}||), \nonumber \\
			&\hspace{4em}3(||F_1||+\ldots + ||F_{k-3}||),\nonumber \\
			&\hspace{4em}\qquad \vdots\nonumber \\
			&\hspace{4em}(k-1)||F_1||\}\nonumber \\
			&\leq \max_{1\leq i \leq k-1} (k-i) \sum_{j=1}^i \norm{F_j}.
		\end{align}
		For the quadratic part,
		\begin{align}\label{eq:upperbound_F2tilde}
		||\tilde{F}_2|| &= \max \{ ||A^1_k||, \nonumber\\
			&\hspace{4em}||A^2_k|| + ||A^2_{k+1}|| ,\nonumber\\
			&\hspace{4em}\qquad \vdots \nonumber\\
			&\hspace{4em}||A^{k-1}_k|| +  ||A^{k-1}_{k+1}|| + \ldots + ||A^{k-1}_{2(k-1)}|| \} \nonumber\\ 
		&\leq \max \left\{ ||F_k||, 
		2(||F_k|| + ||F_{k-1}||),  \ldots \right. \nonumber\\
			&\hspace{4em}\left. (k-1)(||F_k|| +  ||F_{k-1}|| + \ldots + ||F_{2}|| \right\} \nonumber\\ 
		&\leq (k-1) \sum_{j=2}^k \norm{F_j}.
		\end{align}
	\end{proof}
\end{proposition}
Proposition \ref{prop:quadratic_reduction} gives upper bounds on the coefficients of the reduced quadratic system. In practice, the norms of $\Vert \tilde{F}_1 \Vert$ and $\Vert \tilde{F}_2 \Vert$ can be computed exactly using the expressions \eqref{eq:upperbound_F1tilde} and \eqref{eq:upperbound_F2tilde} respectively.

%% file: main_results.tex
Consider the polynomial quadratic system
\begin{equation}
x'= F_1 x+ F_2 x^{[2]},\qquad x \in \R^n. \label{eq:QuadSystemS4}
\end{equation}
Using Proposition \ref{prop:xiDynamics}, we obtain the infinite-dimensional Carleman embedding
\begin{equation}
y'_i =  A^i_i y_i +  A^i_{i+1} y_{i+1}, \qquad i \in \N. \label{eq:QuadSystemInfDim}
\end{equation}
Define the truncated system of order $N$ as
\begin{equation}
\hat{y}'_i =  A^i_i \hat{y}_i +  A^i_{i+1} \hat{y}_{i+1} \delta_{i < N},\qquad 1 \leq i \leq N, \label{eq:QuadSystemTrunc}
\end{equation}
where $ \delta_{i < N}$ is equal to $1$ if $i < N$ and $0$ otherwise.

\begin{definition}
Let the \emph{error of the $i$-th block} be 
\begin{equation}
\eta_i(t) = y_i(t) - \hat{y}_i(t), \qquad 1 \leq i \leq N. \label{eq:errorithblock}
\end{equation}
We also introduce the special notation 
\begin{equation*}
x(t) := y_1(t),\qquad \hat{x}(t) := \hat{y}_1(t),
\end{equation*}
for the solution of the exact and truncated systems respectively, projected onto $\R^n$, and the associated \emph{error on the first block}, or simply the \emph{error}, as
\begin{equation*}
\varepsilon(t) := \eta_1(t) = x(t) - \hat{x}(t).
\end{equation*}

\end{definition} 
Two approaches leading to explicit bounds on $\norm{\varepsilon(t)}$ are considered. In Section~\ref{ssec:errorBoundI} we integrate the differential equation satisfied by the error, to obtain a bound by an explicit integral computation. This formula requires an a priori bound on the norm of the exact solution.
Such estimates can be derived in the general case (see Section~\ref{ssec:Relationship}) but are generally very pessimistic. However,
the system may satisfy some bounds on its solution by construction, especially systems modelling physical
processes.
In Section~\ref{ssec:errorBoundII} we present the result of another method, that exploits the analyticity properties of solutions, and we obtain estimates for the error series using generating functions. These two bounds are compared in Section~\ref{ssec:Relationship}. We close this section with an illustrative application in Section~\ref{ssec:Example}.

\subsection{Backwards integration method} \label{ssec:errorBoundI}

We first consider an error bound based on an a priori estimate of the norm of the exact solution, $x(t)$. 

\begin{theorem} \label{theorem:ErrorBound1}
Let $x : I \to \R^n$ be a solution of the quadratic system \eqref{eq:QuadSystemS4}. Let $\alpha > 0$ be such that
\begin{equation}
\alpha \geq \norm{x}_t := \sup_{\tau \in [0,t]} \norm{x(\tau)}.  \label{eq:xtdefi}
\end{equation}
Then, the error $\varepsilon(t) = x(t) - \hat{x}(t)$ on the solution obtained by Carleman linearization truncated at order $N$ satisfies the estimate
\begin{equation}
\norm{\varepsilon(t)}\leq  \E_1(t) := \dfrac{\alpha^{N+1} \norm{F_2}^N}{\mu(F_1)^N} (e^{\mu(F_1)t} - 1)^N. \label{eq:Thm1MainEstimate}
\end{equation}
If $\mu(F_1) < 0$ then the estimate holds for all $t \geq 0$ and the error converges to $0$. Otherwise, on the interval 
\begin{equation}\label{eq:ThmErrorBnd1:interval}
	0 < t < \frac{1}{\mu(F_1)} \ln \left(1+ \frac{\mu(F_1)}{\norm{x}_t \norm{F_2}} \right)	
\end{equation}
the solution of the truncated system converges, that is, $\lim\limits_{N\to \infty} \norm{\varepsilon(t)} = 0$.
Note that when $\mu(F_1)=0$, the right-hand side of \eqref{eq:ThmErrorBnd1:interval}
is defined by continuity and its value is $\frac{1}{\norm{x}_t\norm{F_2}}$.
\end{theorem}

The proof of this result is presented in Section \ref{ssec:errorBoundIProof}.

\subsection{Power  series method} \label{ssec:errorBoundII}

Now we consider a refined version of Theorem \ref{theorem:ErrorBound1}, where only the initial condition is required (instead of a priori estimates on the norm of the solution, see \eqref{eq:xtdefi}). 

\begin{theorem}  \label{theorem:ErrorBound2}
Let $x : I\to \R^n$ be a solution of the quadratic system \eqref{eq:QuadSystemS4}, and 
\begin{equation}
\beta_0 := \frac{||x_0||~||F_2||}{||F_1||}. \label{eq:beta0defi}
\end{equation}
Then, the error $\varepsilon(t) = x(t) - \hat{x}(t)$ on the solution obtained by Carleman linearization truncated at order $N$ satisfies the estimate
\begin{equation}
||\varepsilon(t)|| \leq \E_2(t) := \dfrac{||x_0|| e^{||F_1||t}}{(1+\beta_0) - \beta_0 e^{||F_1||t}}\left[ \beta_0 (e^{||F_1|| t} - 1 )\right]^N. \label{eq:Thm2MainEstimate}
\end{equation}
Moreover, for all $0 < t < T^*$, where
\begin{equation}
T^* := \frac{1}{||F_1||} \ln \left(1 +  \frac{1}{\beta_0} \right), \label{eq:Thm2Cond}
\end{equation}
the solution of the truncated system converges, that is,
\begin{equation*}
\lim_{N\to \infty} ||\varepsilon(t)|| = 0, \qquad \text{ for all }  t < T^*.
\end{equation*}

\end{theorem}

The proof of this result is presented in Section \ref{ssec:errorBoundIIProof}.

\subsection{Relationship between the error bounds}\label{ssec:Relationship}

It is not immediately clear which of \eqref{eq:Thm1MainEstimate} or \eqref{eq:Thm2MainEstimate} is best. At a first glance, \eqref{eq:Thm1MainEstimate} looks better but requires to know $\alpha$,
which can be really large. On the other hand, \eqref{eq:Thm2MainEstimate} only depends
on the initial condition $x_0$ but is substantially more complicated. In this section,
we derive a generic bound on $\alpha$ based on $t$ and $x_0$ and plug it in
\eqref{eq:Thm1MainEstimate}. We can then compare the two bounds in the specific
situation where we have no a priori bound on $\alpha$. We need this intermediate result.

\begin{proposition}
	If $x : I\to \R^n$ is a solution of the quadratic system $x' = F_1 x + F_1 x^{[2]}$, with initial condition $x(0) = x_0$, then
	\begin{equation}\label{eq:bound_growth_quadratic}
	||x(t)|| \leq \dfrac{||x_0|| e^{||F_1|| t} ||F_1||}{||F_1||+||F_2||(1-e^{||F_1||t})||x_0||}.
	\end{equation}
	\begin{proof}
		Observe that if $p(x) = F_1 x + F_2 x^{[2]}$, then
		\begin{equation}
		||p(x)|| \leq ||F_1|| ~ ||x|| + ||F_1|| ~ ||x||^2, \label{eq:unormodepe}
		\end{equation}
		where use was made of \eqref{eq:ithpowerIdentity}. Define the following differential equation, for $a,b \in \R^+$:
		\begin{equation}
		u' = a u + b u^2,\qquad u(0) = u_0. \label{eq:unormode}
		\end{equation}
		We can easily find an explicit formula for $u(t)$ by the method of separation of variables, obtaining
		\begin{equation}
		u(t) = \dfrac{u_0 a e^{at}}{a+b(1-e^{at})u_0}. \label{eq:unormodeExplicit}
		\end{equation}
		If $u_0$ is chosen such that $u_0 \geq ||x_0||$, and if we choose $a=||F_1||$ and $b = ||F_2||$, then by standard differential inequality arguments we can deduce from \eqref{eq:unormodepe} that the estimate $u(t) \geq ||x(t)||$ holds for all $t \in I$ . Finally using \eqref{eq:unormodeExplicit}, with $u_0 = ||x_0||$, we obtain the claim.
	\end{proof}
\end{proposition}

Plugging \eqref{eq:bound_growth_quadratic} in \eqref{eq:Thm1MainEstimate} we get that
\begin{align*}
\E_1(t) &\leq  \dfrac{\alpha^{N+1} \norm{F_2}^N}{\norm{F_1}^N} (e^{\norm{F_1}t} - 1)^N\\
&\leq \left(\dfrac{||x_0|| e^{\norm{F_1} t} \norm{F_1}}{\norm{F_1} + \norm{F_2}(1-e^{\norm{F_1}t})\norm{x_0}}\right)^{N+1}
\frac{\norm{F_2}^N}{\norm{F_1}^N}(e^{\norm{F_1}t}-1)^N\\
&\leq \left(\dfrac{e^{\norm{F_1} t}}{1+\frac{\norm{F_2}~\norm{x_0}}{\norm{F_1}}(1-e^{\norm{F_1} t})}\right)^{N+1}
\frac{\norm{x_0}^{N+1}\norm{F_2}^N}{\norm{F_1}^N}(e^{\norm{F_1} t}-1)^N\\
&\leq \left(\dfrac{e^{\norm{F_1} t}}{1+\beta_0(1-e^{\norm{F_1} t})}\right)^{N+1}
\norm{x_0}\beta_0^N(e^{\norm{F_1} t}-1)^N\\
&\leq \left(\dfrac{e^{\norm{F_1} t}}{(1+\beta_0)-\beta_0 e^{\norm{F_1} t}}\right)^{N+1}
\norm{x_0}\left[ \beta_0 (e^{\norm{F_1} t} - 1 )\right]^N\\
&\leq \left(\dfrac{e^{\norm{F_1} t}}{(1+\beta_0)-\beta_0 e^{\norm{F_1} t}}\right)^N \E_2(t).
\end{align*}
It is thus clear that if we simply use the worst case bound on $\alpha$, $\E_1(t)$
can be significantly worse than $\E_2(t)$, possibly by an exponential factor in $N$.
This suggests that $\E_1(t)$ is only useful if we have an a priori bound on $\alpha$
that is much better than the worst case. Finally note that $\E_1(t)$ can be valid for
much longer time intervals than $\E_2(t)$ because the existence of $\alpha$ implies
the existence of the solution, something that $\E_2(t)$ cannot capture.

\subsection{Example} \label{ssec:Example}

We have implemented Carleman linearization of polynomial ODEs in our software package \texttt{carlin}, which is publicly available \cite{FO17carlin}. It is written in Python, and for the symbolic polynomial manipulations we rely on the open-source mathematics sofware system SageMath \cite{sagemath}. For the numerical computations we use sparse matrix linear algebra provided by SciPy \cite{scipy}.

\begin{figure}[!ht] \label{fig:vanderpol}
	\centering
	\subfigure[Phase portrait]{
		\includegraphics[height=220pt]{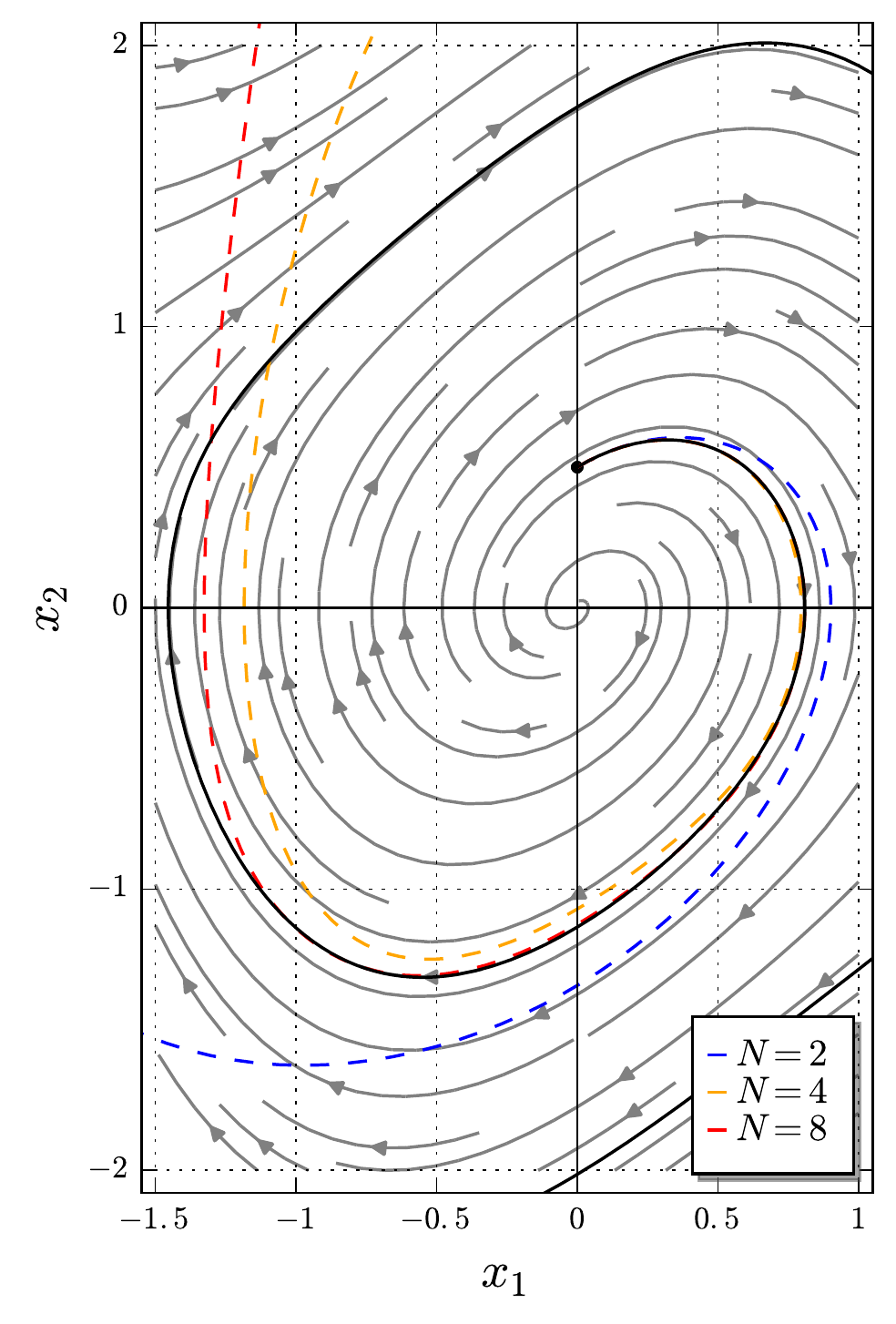}\label{fig:vanderpol_pp}
	} 
	\subfigure[Evolution of $x_1(t)$ as a function of time]{
		\includegraphics[height=210pt]{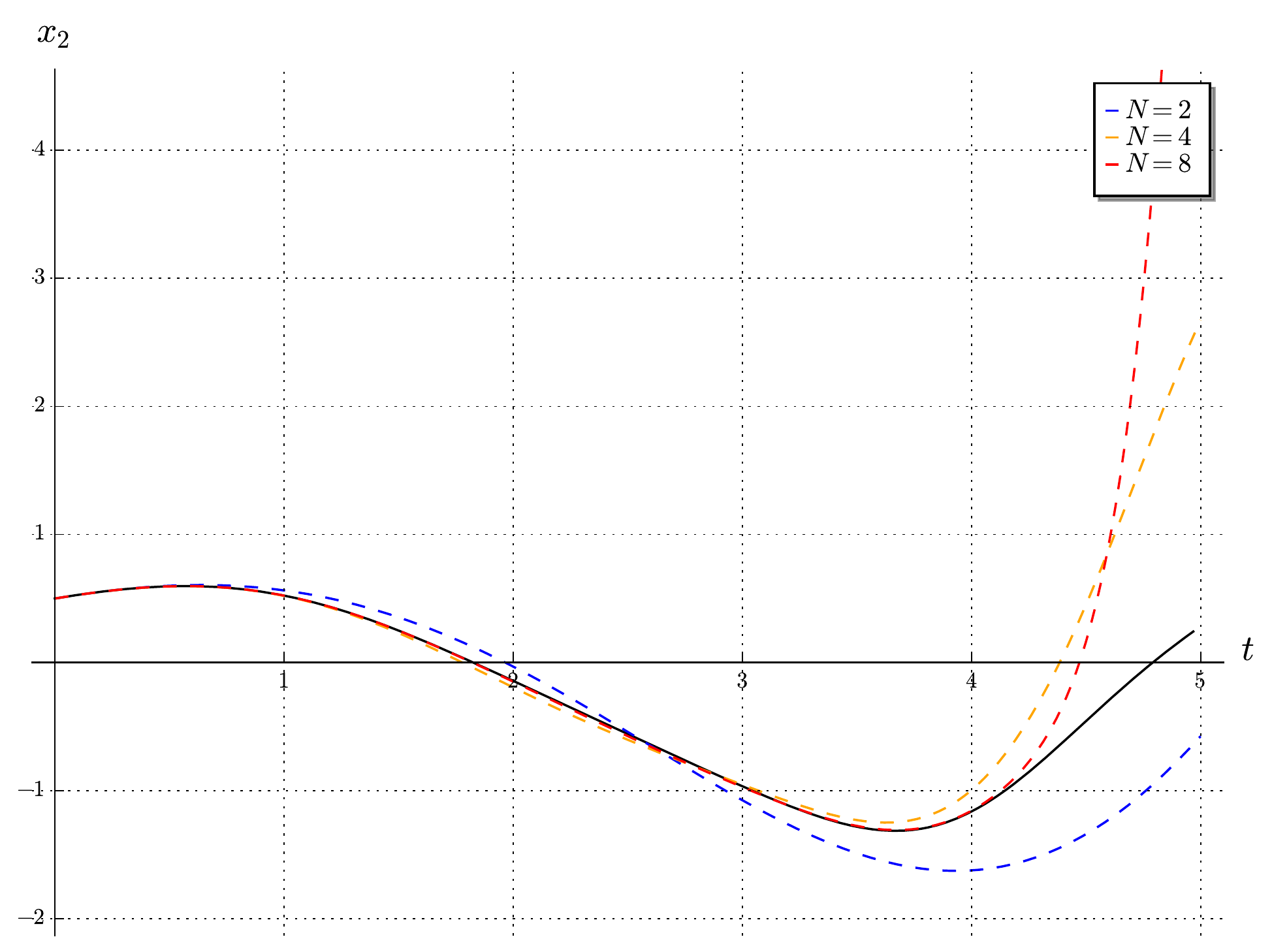}\label{fig:vanderpol_x1t}
	}	
	\subfigure[Error bound envelopes]{
		\includegraphics[height=210pt]{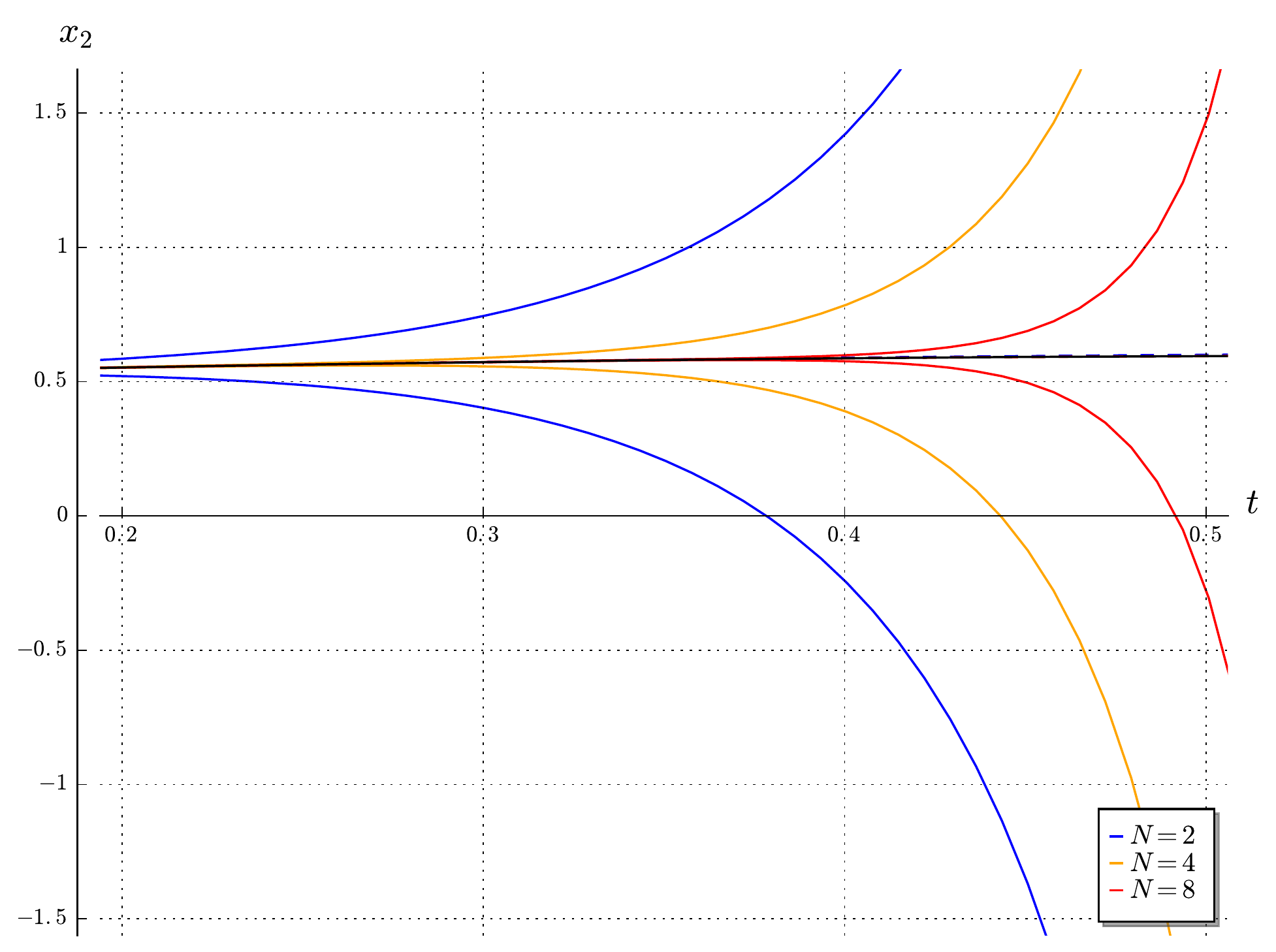}\label{fig:vanderpol_errorbnd}
	}
	\caption{Simulation of the Van der Pol oscillator: solution obtained with Carleman linearization truncated at different orders (dashed, orders 2, 4 and 8) and the solution to the nonlinear ODE (solid line). The initial condition is $x_0 = (0, 0.5)^\transp$ and the parameters are $r = 0.6$ and $\omega=1$.}
\end{figure}

As an illustrative example, consider the Van der Pol oscillator, which is a non-conservative system with non-linear damping, described by the equations
\begin{equation} \label{eq:vanderpol}
\left\{ 
\begin{aligned}
x'_1 &= x_2 \\ 
x'_2 &= - \omega^2 x_1   + r(1-x_1^2)x_2 
\end{aligned}
\right.
\end{equation}
Here $\omega>0$ is the natural frequency of the oscillator, and $r > 0$ is a scalar parameter indicating the damping factor. Setting $x := (x_1,x_2)^\transp$, system \eqref{eq:vanderpol} written in the standard ODE form \eqref{eq:S_ode} is 
$$
x'(t) = F_1 x(t) + F_3 x^{[3]}(t),
$$ 
with
$F_1 \in \mathbb{R}^{2\times 2}$ given by
\begin{equation*}
F_1 = \begin{pmatrix}
0 & 1 \\ -\omega^2 & r
\end{pmatrix},
\end{equation*}
and $F_3 \in \mathbb{R}^{2\times 8}$
\begin{equation*}
F_3 = \begin{pmatrix}
0 & 0 & 0 & 0 & 0 & 0 & 0 & 0 \\
0 & -r & 0 & 0 & 0 & 0 & 0 & 0 
\end{pmatrix}. 
\end{equation*}
(the quadratic term is identically zero for this system, $F_2 = \mathbb{O}_{2\times 4}$). The supremum norms are $||F_1|| = \max \{1, \omega^2+r\}$ and $||F_3|| = r$ respectively.

In Figure \ref{fig:vanderpol_pp} we plot the solution of the finite-dimensional linear system, \eqref{eq:TruncSysMatrix}, for different values of truncation order $N$. For validation, we plot the solution of the nonlinear system \eqref{eq:vanderpol} obtained by a  4th order classical Runge-Kutta method. In Figure \ref{fig:vanderpol_x1t} we show the coordinate $x_2(t)$ as a function of time. Increasing the order $N$ improves the quality of the approximation on a longer time interval, but at the same time, it diverges faster to infinity closer to the range of validity of the approximation.

The error bound from Theorem \ref{theorem:ErrorBound2} is represented in Figure \ref{fig:vanderpol_errorbnd} for different values of $N=2,4$ and $8$, which is summed to the actual solutions and we take $\hat{x}_2(t) \pm \mathcal{E}_2(t)$. The convergence radius of the error formula for this choice of parameters is $T^* \approx 0.58$. We observe that the error bound provides an enclosing envelope for the solution. This bound is conservative, as it is clear by comparison to the actual evolution of the linearized solutions at different orders from Figure \ref{fig:vanderpol_x1t}.

%% file: proofs.tex
\subsection{Proof of Theorem 1} \label{ssec:errorBoundIProof}

In principle we can find an explicit formula for $\varepsilon(t)$ by a straightforward integration of the ODE satisfied by the errors, 
\begin{equation}
\eta'_i(t) = A^i_i \eta_i(t) + A^i_{i+1}\left( y_{i+1}(t)-\hat{y}_{i+1}(t)\delta_{i<N} \right), \qquad 1 \leq i \leq N, \label{eq:errorseq}
\end{equation}
obtained by differentiating \eqref{eq:errorithblock} and substituting with \eqref{eq:QuadSystemInfDim} and \eqref{eq:QuadSystemTrunc}. However, the coupling at different $i$ makes this computation cumbersome. A better approach, similar to the one in \cite{bellman1962some} for the scalar case, is to systematically use backward-substitution. This consists of integrating \eqref{eq:errorseq} for decreasing $i$, starting from $i=N$, then $i=N-1$, until $i=1$.  To proceed further it is convenient to introduce the following function and a time-dependent norm estimate.
\begin{lemma} \label{lemma:Hestimate}
For each $i=1,\ldots,N$ and $t,s \in \R^+$, $0\leq s \leq t$, let 
\begin{equation*}
H_{i}(t,s) := e^{A^i_i (t-s)} A^i_{i+1}.
\end{equation*}
Then
\begin{equation}
||H_i(t,s)|| \leq i e^{i\mu(F_1)(t-s)} ||F_2||. \label{eq:EstimateHi}
\end{equation}
\begin{proof}
Using the submultiplicativity property of the supremum norm together with Proposition \ref{prop:PropAijFjestimate}, it follows that
\begin{equation*}
||H_i(t,s)|| \leq ||e^{A^i_i(t-s)}|| ~ ||A^i_{i+1}|| \leq i e^{i\mu(F_1)(t-s)} ||F_2||. 
\end{equation*}
\end{proof}
\end{lemma}
The explicit computation of the following multiple integral is relegated to Appendix \ref{app:ProofPropIntegrals}. 
\begin{lemma} \label{lemma:RecurrentIntegralTN}
For all $N\geq 1$, and $s_N > 0$, 
\begin{equation}
\int_{0}^{s_N} \cdots \int_{0}^{s_2} \int_{0}^{s_1} e^{a \left( -Ns_0 +  \sum\limits_{i=1}^{N} s_i \right)} \dd s_0 \dd s_1 \cdots \dd s_{N-1}  = \dfrac{(e^{a s_N}-1)^N}{N!a^{N}}. \label{eq:LemmaRecurrentIntegralTN}
\end{equation}
The result holds for all $a \in\R$, and the right-hand side for $a = 0$ reduces to $\dfrac{s_N^N}{N!}$.
\end{lemma}
The error can be expressed exactly as a nested integral involving the $(N+1)$-th order term of the exact solution.
\begin{proposition}
The error on the first block, $\eta_1(t) = y_1(t)-\hat{y}_1(t)$, is 
\begin{align}
\begin{split}
 \eta_1(t) &=  \int_{0}^t  H_1(t,s_{N-1})  \int_{0}^{s_{N-1}}  H_2(s_{N-1},s_{N-2}) \\ 
&\cdots \int_{0}^{s_2} H_{N-1}(s_{2},s_{1})    \int_{0}^{s_1} H_N(s_{1},s) y_{N+1}(s) \dd s\dd s_{1}\cdots\dd s_{N-1}. \label{eq:E1ExactFormulak2}
\end{split}
\end{align}

\begin{proof}
We proceed by backwards substitution. For $i=N$, we have 
\begin{equation*}
\dfrac{d \eta_N}{dt}(t) = A^N_N \eta_N(t) + A^N_{N+1} y_{N+1}(t).
\end{equation*}
By integration and since $\eta_N(0)=0$, 
\begin{subequations}
\begin{align*}
\eta_N(t) &= \int_0^t e^{A^N_N (t-s)} A^N_{N+1} y_{N+1}(s) \dd s \\
&= \int_0^t H_N(t,s) y_{N+1}(s) \dd s. 
\end{align*}
\end{subequations}
For $i=N-1$, 
\begin{equation*}
\dfrac{d \eta_{N-1}}{dt}(t) = A^{N-1}_{N-1} \eta_{N-1}(t) + A^{N-1}_N \eta_N(t).
\end{equation*}
Again integrating and using that $\eta_{N-1}(0)=0$, 
\begin{subequations}
\begin{align*}
\eta_{N-1}(t) &= \int_0^t e^{A^{N-1}_{N-1} (t-s_1)} A^{N-1}_N \eta_N(s_1) \dd s_1 \\
&= \int_0^t H_{N-1}(t,s_1) \eta_N(s_1) \dd s_1 \\
&= \int_0^t H_{N-1}(t,s_1) \left( \int_0^{s_1} H_N(s_1,s) y_{N+1}(s) \dd s \right) \dd s_1 \\
&= \int_0^t H_{N-1}(t,s_1) \int_0^{s_1} H_N(s_1,s) y_{N+1}(s) \dd s\dd s_1.
\end{align*}
\end{subequations}
Iterating this procedure until $i=1$ we recover formula \eqref{eq:E1ExactFormulak2}.
\end{proof}

\end{proposition}

\begin{proof}[Proof of Theorem \ref{theorem:ErrorBound1}.] 
We start from \eqref{eq:E1ExactFormulak2}, taking norms on both sides, then 
\begin{align*}
\begin{split}
 ||\eta_1(t)|| &\leq \int_{0}^t \norm{H_1(t, s_{N-1})} \int_{0}^{s_{N-1}}  ||H_2(s_{N-1},s_{N-2})|| \cdots  \\ 
&\cdots \int_{0}^{s_2} ||H_{N-1}(s_{2},s_{1})||    \int_{0}^{s_1} ||H_N(s_{1},s)|| ~||y_{N+1}(s)|| \dd s\prod_{i=1}^{N-1} \dd s_i.
\end{split}
\end{align*}
If $\norm{x}_t = \sup_{s \in [0,t]} \norm{x(s)}$ and using \eqref{eq:ithpowerIdentity} and \eqref{eq:changeofvariables}, it follows that $||y_{N+1}(s)|| \leq ||x||_t^{N+1}$ for all $s \in [0,t]$, and from Lemma \ref{lemma:Hestimate},
\begin{align*}
\begin{split}
\norm{\varepsilon(t)} &\leq ||F_2||^N \int_{0}^t  e^{\mu(F_1)(t-s_{N-1})}   \int_{0}^{s_{N-1}}  2 e^{2\mu(F_1)(s_{N-1} - s_{N-2})}   \cdots  \\ 
&\cdots \int_{0}^{s_2} (N-1) e^{(N-1)\mu(F_1)(s_2 - s_1)} \int_{0}^{s_1} N e^{N\mu(F_1)(s_1 - s)}\norm{x}_t^{N+1} \dd s\prod_{i=1}^{N-1} \dd s_i \\
&= \norm{x}_t^{N+1} \norm{F_2}^N N! ~ G_N(\mu(F_1), t),
\end{split}
\end{align*}
where we have conveniently defined
\begin{align}
\begin{split}
G_N(a, t) &:= \int_{0}^t  e^{a (t-s_{N-1})}   \int_{0}^{s_{N-1}}  e^{2a(s_{N-1} - s_{N-2})} \\
&\cdots  \int_{0}^{s_2}  e^{(N-1)a(s_2 - s_1)}   \int_{0}^{s_1} e^{Na(s_1 - s)} \dd s\prod_{i=1}^{N-1} \dd s_i.
\end{split}
\end{align}
We apply Lemma \ref{lemma:RecurrentIntegralTN}, first by renaming $t \to s_N$ and $s \to s_0$, then setting $a = \mu(F_1)$ and $s_N = t$,
\begin{subequations}
\begin{align*}
G_N(a, s_N) &= \int_{0}^{s_N}   \cdots \int_{0}^{s_1} e^{a \sum\limits_{i=0}^{N-1} (N-i) (s_{i+1} - s_i ) } \prod_{i=0}^{N-1} \dd s_i \\ 
&= \int_{0}^{s_N}    \cdots  \int_{0}^{s_1} e^{a \left( -Ns_0 +  \sum\limits_{i=1}^{N} s_i \right)} \prod_{i=0}^{N-1} \dd s_i \\
&= \frac{(e^{a t} - 1 )^N}{N! a^{N}},
\end{align*}
\end{subequations}
Note that the last equality holds even for $a=0$, where the right-hand side exists by continuity.
 Choosing $\alpha > 0$ such that $\alpha \geq \norm{x}_t$, we obtain the formula  
\begin{equation}
\norm{\varepsilon(t)} \leq \frac{\alpha^{N+1} \norm{F_2}^N}{\mu(F_1)^N} (e^{\mu(F_1) t} - 1)^N, \label{eq:estimatee1t1}
\end{equation} 
as claimed. where the value for $\mu(F_1)=0$ is defined by continuity.

To find the radius of convergence, we use the estimate \eqref{eq:estimatee1t1}. If $\mu(F_1)\neq0$, let us rearrange the right-hand side of \eqref{eq:estimatee1t1} as $$\norm{\varepsilon(t)} \leq \gamma_t \left[ \gamma_t  (e^{\norm{(F_1)} t} - 1)\right]^N,$$ where $\gamma_t := \frac{ \norm{x}_t \norm{F_2}}{\mu(F_1)}$.  The right-hand side converges to zero as $N\to \infty$ provided that the term in square brackets, which is non-negative, has modulus strictly smaller than $1$, that is, $\gamma_t (e^{\mu(F_1) t} - 1) < 1$, and the formula for the radius of convergence follows.
Finally, when $\mu(F_1)=0$, estimate \eqref{eq:estimatee1t1} becomes
\begin{equation}
\norm{\varepsilon(t)} \leq \alpha^{N+1} \norm{F_2}^Nt^N, \label{eq:estimatee1t2}
\end{equation}
which converges when $\alpha\norm{F_2}t<1$. The obtained bound matches exactly the
limit value of the formula in the case where $\mu(F_1)\neq 0$, thus we can use the
same bound for all cases.
\end{proof}

\subsection{Proof of Theorem 2} \label{ssec:errorBoundIIProof}

The idea of the proof is to construct a recurrence for the error term, and majorate it by a linear recurrence inequality. Then, we explicitly solve this linear recurrence inequality by the method of generating functions.

\subsubsection{Path sums}

Let us develop the analytic solutions of \eqref{eq:QuadSystemInfDim},
\begin{equation}
y_i(t) = \sum_{\nu = 0}^\infty \chi_{i,\nu} \frac{t^\nu}{\nu!},\qquad i \in \N, \label{eq:yiAnalyticDev}
\end{equation}
with an initial condition compatible with the embedding, i.e. $y_i(0) = x_0^{[i]}$, $i\geq 1$, and
\begin{equation}
\chi_{i,\nu} := \frac{d^\nu y_i}{d t^\nu}(0) , \qquad i\geq 1,~\nu \geq 0, \label{eq:yichicoeff}
\end{equation}
where $\dim \chi_{i,\nu} = n^i$ for all $\nu \geq 0$. By convention $\nu=0$ corresponds to the function itself, that is, we set $y_i^{(0)} (t) := y_i(t)$. The next step is to work out the coefficients $\chi_{i,\nu}$, by taking higher order derivatives of \eqref{eq:QuadSystemInfDim}. To build some intuition, consider an example.

\begin{example}
For the second order derivative, we differentiate \eqref{eq:QuadSystemInfDim}, obtaining
\begin{align}
y''_i&= A^i_i y'_i + A^i_{i+1} y'_{i+1} \nonumber\\
&= A^i_i \left( A^i_i y_i + A^i_{i+1} y_{i+1} \right) + A^i_{i+1} \left( A^{i+1}_{i+1} y_{i+1} + A^{i+1}_{i+2} y_{i+2}\right) \nonumber \\
&= \left( A^i_i  A^i_i \right) y_i + \left( A^i_i A^i_{i+1} + A^i_{i+1} A^{i+1}_{i+1} \right) y_{i+1} + \left( A^i_{i+1}A^{i+1}_{i+2} \right)  y_{i+2}. \label{eq:d2yEx1}
\end{align}
The terms in each different order $i$, $i+1$ and $i+2$ have been grouped, because we shall associate these terms to their corresponding path sums as defined below.
\end{example} 
First we need to define what is a single path.

\begin{definition} \label{defi:pathco}
A \emph{jump} between sites $i$ and $i+j$, for $i\geq 1$, $j\geq 0$, is the linear map $A^{i}_{i+j} : \R^{n^{i+j}} \to \R^{n^i}$. The \emph{length} of a jump is the number of sites travelled to the right. For example, the length of the jump $A^{i}_{i+j}$ is $j$. A \emph{path} between sites $i$ and $i+j$ of $\nu$ jumps and order $q$ is an ordered sequence of products of matrices
\begin{equation*}
\joli{P}_{i,i+j}^{(\nu,q)}(\alpha) = \underset{\nu \text{ factors}}{\underbrace{A^{\alpha_1}_{\alpha_2}A^{\alpha_2}_{\alpha_3}\cdots A^{\alpha_{m-2}}_{\alpha_{m-1}} A^{\alpha_{m-1}}_{\alpha_m}}},
\end{equation*} 
where $i\geq 1$, $j\geq 0$, $\nu \geq 0$, $q \geq 0$, and where $\alpha = (\alpha_1,\ldots,\alpha_m)$ is a multi-index of ordered integers from $i$ to $i+j$, i.e. $i = \alpha_1 \leq \alpha_2 \leq \ldots \leq  i+j = \alpha_m$. By convention, the empty product $\nu=0$ is defined as the identity.
Finally, we say that \emph{order} of the path is $q = \max \{\alpha_2-\alpha_1,\ldots, \alpha_{m}-\alpha_{m-1}\}$. It corresponds to the maximal individual jump length attained in the path. 
\end{definition}

Now we turn into the definition of a (combinatorial) path sum. 
\begin{definition} \label{defi:copathsum}
The \emph{path sum} $C^{(\nu,q)}_{i,i+j}$ is defined as
\begin{equation}
C^{(\nu,q)}_{i,i+j} = \sum_{\alpha} \joli{P}^{(\nu,q)}_{i,i+j}(\alpha), \label{defi:PathSum}
\end{equation}
where the sum is taken over all paths of $\nu$ jumps and order $q$ between sites $i$ and $i+j$. Moreover, when it is understood that $q=1$ is fixed, we set 
\begin{equation}
C^{(\nu)}_{i,i+j} := C^{(\nu,1)}_{i,i+j}.  \label{defi:copathsumSimple}
\end{equation}
\end{definition}

\addtocounter{example}{-1}
\begin{example}[continuation] 
We can write the computation \eqref{eq:d2yEx1} in condensed form as
\begin{subequations}
\begin{align*}
y_i'' &= C_{i,i}^{(2)} y_i + C_{i,i+1}^{(2)} y_{i+1} + C_{i,i+2}^{(2)}  y_{i+2}. 
\end{align*}
\end{subequations}
The path sums are pictorially represented as

\begin{equation*}
\begin{tikzpicture}[->,>=stealth',shorten >=1pt,auto,node distance=2.2cm,
semithick]
\tikzstyle{every state}=[fill=white,draw=black,text=black,inner sep=0pt,minimum size=1cm]

\node (V) {\large $C_{i,i}^{(2)} ~~ \equiv$};
\node[state] (i) [right of=V,node distance=1.8cm]{$i$};

\path (i) edge[loop above] 	node {$A^i_i$} (i1);
\path (i) edge[loop right] 	node {$A^i_i$} (i2);

\end{tikzpicture}
\end{equation*}
Similarly for the other terms, 
\begin{equation*}
\begin{tikzpicture}[->,>=stealth',shorten >=1pt,auto,
semithick]
\tikzstyle{every state}=[fill=white,draw=black,text=black,inner sep=0pt,minimum size=1cm]

\node (V) {\large $C_{i,i+1}^{(2)} ~~ \equiv$};
\node[state](A)   [right of=V,node distance=2.2cm]   {$i$};
\node[state] (C) [right of=A,node distance=2.2cm]  {$i+1$};

\path[->] (A) edge[bend left]              node {$A^i_{i+1}$} (C)
edge [loop above] node {$A^i_i$} (A);

\node (W) [right of=C,node distance=1.1cm] {\large $+$};
\node[state](B)   [right of=W,node distance=1.1cm]   {$i$};
\node[state] (D) [right of=B, node distance=2.2cm]  {$i+1$};
\path[->] (B) edge[bend left]              node {$A^i_{i+1}$} (D);
\path[->] (D) edge [loop above] node {$A^{i+1}_{i+1}$} (D);

\end{tikzpicture}
\end{equation*}
and
\begin{equation*} 
\begin{tikzpicture}[->,>=stealth',shorten >=1pt,auto,
semithick]
\tikzstyle{every state}=[fill=white,draw=black,text=black,inner sep=0pt,minimum size=1cm]

\node (V) {\large $C_{i,i+2}^{(2)} ~~ \equiv$};
\node[state](A)   [right of=V,node distance=2.2cm]   {$i$};
\node[state] (C) [right of=A,node distance=2.2cm]  {$i+1$};
\node[state] (D) [right of=C,node distance=2.2cm]  {$i+2$};
\path[->] (A) edge[bend left]              node {$A^i_{i+1}$} (C);
\path[->] (C) edge[bend left]              node {$A^{i+1}_{i+2}$} (D);
\node[state](B)   [right of=W,node distance=1.1cm]   {$i+2$};
\end{tikzpicture}
\end{equation*}
 
\end{example}

\begin{remark}

A path $\joli{P}_{i,i+j}^{(\nu,q)}(\alpha)$ is non-empty only for $0 \leq j \leq \nu q$. In consequence, $C^{(\nu,q)}_{i,i+j}$ is zero for $j>\nu q$. In other words, to go from site $i$ to site $i+j$ we need to travel a distance $j$ to the right, and the maximum we can travel is taking all $\nu$ paths of the same maximal length $q$. This is illustrated in the diagram below:
\begin{equation*}
\begin{tikzpicture}[->,>=stealth',shorten >=1pt,auto,
semithick]
\tikzstyle{every state}=[fill=white,draw=black,text=black,inner sep=0pt,minimum size=1cm]

\node[state](A)   [left of=B,node distance=2.2cm]   {$i$};
\path[->] (A) edge[bend left]          node {$A^i_{i+q}$} (B);

\node[state] (B) [right of=A,node distance=2.2cm]  {$i+q$};
\path[-, loosely dotted, line width=2pt] (A) edge[]              node {} (B);

\node[state](C)   [right of=B,node distance=2.8cm]   {$\kappa$};
\path[-, loosely dotted, line width=2pt] (B) edge[]              node {} (C);

\node[state](D)   [right of=C,node distance=2.2cm]   {$\kappa+q$};
\path[-, loosely dotted, line width=2pt] (C) edge[]              node {} (D);

\path[->] (C) edge[bend left]              node {$A^{\kappa}_{\kappa+q}$} (D);

\end{tikzpicture}
\end{equation*}
where $\kappa$ is a shortcut for $i + (\nu-1)q$.
\end{remark}

\subsubsection{Recurrence for the path sums}

Next we explore the recurrence relation satisfied by the path sums. By hypothesis the ODE is quadratic ($k=2$), hence we can fix $q=1$ and, as we did above with the examples, only write the index corresponding to the number of jumps, $\nu$. 

\begin{proposition} \label{prop:recproofpsk2}
Let $i\geq 1$,  $\nu \geq 2$ and $0 \leq j \leq \nu$, and $q=1$. Then:
\begin{enumerate}
\item If $1 \leq j \leq \nu-1$, then
\begin{equation*}
C^{(\nu)}_{i,i+j} = C^{(\nu-1)}_{i,i+j}C^{(1)}_{i+j,i+j} + C^{(\nu-1)}_{i,i+j-1}C^{(1)}_{i+j-1,i+j}. \label{eq:FundRecPathsum}
\end{equation*}
Note that $C^{(1)}_{i+j,i+j}=A^{i+j}_{i+j}$, and that $C^{(1)}_{i+j-1,i+j} = A^{i+j-1}_{i+j}$. 
\item For $j=0$, $C^{(\nu)}_{i,i+j} = C^{(\nu)}_{i,i} = \underset{\nu \text{ times}}{\underbrace{A^i_i A^i_i \cdots A^i_i }}$. 
\item For $j=\nu$, $C^{(\nu)}_{i,i+j} = C^{(\nu)}_{i,i+\nu} = \underset{\nu \text{ terms}}{\underbrace{  A^i_{i+1} A^{i+1}_{i+2} \ldots A^{i+\nu-1}_{i + \nu }}}$.
\end{enumerate}

\begin{proof}
The extremal cases $j=0$ and $j=\nu$ are trivial. For the general recurrence, assume that $1 \leq j \leq \nu-1$, and observe that
\begin{equation*}
\begin{tikzpicture}[->,>=stealth',shorten >=1pt,auto,
semithick]
\tikzstyle{every state}=[fill=white,draw=black,text=black,inner sep=0pt,minimum size=1cm]

\node[state](A)   [left of=B,node distance=2.2cm]   {$i$};
\path[->] (A) edge[bend left]          node {$C^{(\nu)}_{i, i + j}$} (B);

\node[state] (B) [right of=A,node distance=2.2cm]  {$i + j$};
\path[-, loosely dotted, line width=2pt] (A) edge[]              node {} (B); 

\end{tikzpicture}
\end{equation*}
can be decomposed as:
\begin{equation*}
\begin{tikzpicture}[->,>=stealth',shorten >=1pt,auto,
semithick]
\tikzstyle{every state}=[fill=white,draw=black,text=black,inner sep=0pt,minimum size=1cm]

\node[state](A)   [left of=B, node distance=2.2cm]   {$i$};
\path[->] (A) edge[bend left]          node {$C^{(\nu -1)}_{i, i + j}$} (B);

\path[->] (B) edge[loop above]          node {$A^{i+j}_{i + j}$} (B);

\node[state] (B) [right of=A,node distance=2.2cm]  {$i + j$};
\path[-, loosely dotted, line width=2pt] (A) edge[]              node {} (B); 

\node[state] (C) [right of=B,node distance=1.6cm, draw=none]  {$+$};

\node[state](D)  [right of=C, node distance=1.6cm]   {$i$};

\node[state](E)  [right of=D, node distance=2.2cm]   {$\kappa$};
\path[->] (D) edge[bend left]    node {$C^{(\nu-1)}_{i, \kappa}$} (E);
\path[-, loosely dotted, line width=2pt] (D) edge[]              node {} (E); 

\node[state](F)  [right of=E, node distance=1.8cm]   {$i+j$};

\path[->] (E) edge[bend left]    node {$A^{\kappa}_{\kappa+1}$} (F);

\end{tikzpicture}
\end{equation*}
where $\kappa = i+j-1$.
We remark that the first diagram of the right-hand side is non-vanishing only if $j\leq \nu-1$, that is, if there are enough jumps so that removing one still allows to arrive to the site $i+j$. 

%
\end{proof}
\end{proposition}

In the following proposition we use the combinatorial path sum to express the Taylor coefficients $\xi_{i,\nu}$ (see Eqs. \eqref{eq:yiAnalyticDev}-\eqref{eq:yichicoeff}) in terms of the linear maps $A^{i}_{i+j}$. 

\begin{proposition} \label{prop:derxik2}
For $\nu\geq 0$, $i\geq 1$, the following formula holds
\begin{equation} 
\frac{d^\nu y_i}{dt^\nu}(t) = \sum_{j=0}^\nu C_{i,i+j}^{(\nu)} y_{i+j}(t), \label{eq:devderlk2p}
\end{equation}
where $C_{i,i+j}^{(\nu)}$ is the combinatorial path sum defined in Eq. \eqref{defi:copathsumSimple}. By definition, $C_{i,i}^{(0)} = \Idn^{[i]}$.
\begin{proof}
We already proved the base cases $\nu=0,1,2$. Assume that the inequality holds for $\nu -1 \geq 2$, and for any $i\geq 1$, and let us prove that it holds for $\nu$. By the inductive hypothesis and the recurrence formula proved in Proposition \ref{prop:recproofpsk2},
\begin{subequations}
\begin{align*}
y^{(\nu)}_i &= \sum_{j=0}^{\nu-1} C_{i,i+j}^{(\nu-1)} y'_{i+j} \\ &= \sum_{j=0}^{\nu-1} C_{i,i+j}^{(\nu-1)}\left( A^{i+j}_{i+j} y_{i+j} + A^{i+j}_{i+j+1} y_{i+j+1} \right)  \\
&= \sum_{j=0}^{\nu-1} C_{i,i+j}^{(\nu-1)}  A^{i+j}_{i+j} y_{i+j} + \sum_{j=1}^{\nu}  C_{i,i+j-1}^{(\nu-1)} A^{i+j-1}_{i+j} y_{i+j} \\
&= \sum_{j=1}^{\nu-1} \left( C_{i,i+j}^{(\nu-1)}  A^{i+j}_{i+j} + C_{i,i+j-1}^{(\nu-1)} A^{i+j-1}_{i+j}  \right) y_{i+j}\nonumber  \\ &+ C^{(\nu-1)}_{i,i} A^{i}_{i} y_i + C^{(\nu-1)}_{i,i+\nu-1} A^{i+\nu-1}_{i+\nu} y_{i+\nu} \\
&= \sum_{j=1}^{\nu-1} C^{(\nu)}_{i,i+j} y_{i+j}  + C^{(\nu-1)}_{i,i} A^{i}_{i} y_i + C^{(\nu-1)}_{i,i+\nu-1} A^{i+\nu-1}_{i+\nu} y_{i+\nu} \\
&= \sum_{j=0}^{\nu} C^{(\nu)}_{i,i+j} y_{i+j}.
\end{align*}
\end{subequations}
\end{proof}
\end{proposition}

Using \eqref{eq:yiAnalyticDev}-\eqref{eq:yichicoeff} and Proposition \ref{prop:derxik2}, it follows that
\begin{align}
y_i(t) =  \sum_{\nu=0}^\infty \chi_{i,\nu}  \frac{t^\nu}{\nu!} &= \sum_{\nu=0}^\infty \left. \frac{d^\nu y_i}{dt^\nu} \right\vert_{t=0}   \frac{t^\nu}{\nu!} \nonumber \\ &=
 \sum_{\nu=0}^\infty \left( \sum_{j=0}^\nu C^{(\nu)}_{i,i+j} y_{i+j}(0)  \right) \frac{t^\nu}{\nu!} \nonumber \\ 
&= \sum_{\nu=0}^\infty   \sum_{j=0}^\nu C^{(\nu)}_{i,i+j} x_{0}^{[i+j]} \frac{t^\nu}{\nu!}. \label{eq:yi_analytic}
\end{align}

\subsubsection{Truncation of the power series}

How do we have to modify the path sum if we consider the truncated system \eqref{eq:QuadSystemTrunc}? We have to be careful with the links considered, since it does only make sense to include $C_{i,i+j}$ for $i+j \leq N$, because there are no links further than that. Let us again consider an example to build some intuition. 

\begin{example}
Let $\nu=2$ and $1\leq i \leq N$. Differentiating \eqref{eq:QuadSystemTrunc}, we find that
\begin{subequations}
\begin{align*}
\hat{y}''_i &= A^i_i \hat{y}'_i + A^i_{i+1} \hat{y}'_{i+1}\delta_{i\leq N-1} \\
&= A^i_i \left( A^i_i \hat{y}_i + A^i_{i+1} \hat{y}_{i+1} \delta_{i\leq N-1}\right)  + A^i_{i+1} \left( A^{i+1}_{i+1}  \hat{y}_{i+1} \delta_{i\leq N-1} \right. \\ &\qquad + \left. A^{i+1}_{i+2} \hat{y}_{i+2}\delta_{i\leq N-1}\delta_{i\leq N-2}\right)  \\
&= \left( A^i_i  A^i_i \right) \hat{y}_i + \left( A^i_i A^i_{i+1} + A^i_{i+1} A^{i+1}_{i+1} \right) \hat{y}_{i+1}\delta_{i\leq N-1}  \\ &\qquad+ \left( A^i_{i+1}A^{i+1}_{i+2} \right)  \hat{y}_{i+2}\delta_{i\leq N-2} \\
&= C_{i,i}^{(2)} \hat{y}_i + C_{i,i+1}^{(2)} \hat{y}_{i+1}\delta_{i\leq N-1}+ C_{i,i+2}^{(2)}  \hat{y}_{i+2}\delta_{i\leq N-2}. 
\end{align*}
\end{subequations}
Here we recognise the path sums, although as remarked above, the Kronecker deltas are there to cut some terms from the expansion if $i$ is sufficiently high. 
\end{example}
 
The higher-order derivatives of the truncated system \eqref{eq:QuadSystemTrunc} satisfy
\begin{equation} 
\frac{d^\nu \hat{y}_i}{dt^\nu}(t) = \sum_{j=0}^{\nu} C_{i,i+j}^{(\nu)} y_{i+j}(t)\delta_{i \leq N-j}.  \label{eq:devderlk2pTrunc}
\end{equation}
This is proved in the same way as we did for Proposition \ref{prop:derxik2}. Hence, for arbitrary $1 \leq i \leq N$ we can write
\begin{align}
\hat{y}_i(t) =  \sum_{\nu=0}^\infty \xi_{i,\nu}  \frac{t^\nu}{\nu!} &= \sum_{\nu=0}^\infty \left. \frac{d^\nu \hat{y}_i}{dt^\nu} \right\vert_{t=0}   \frac{t^\nu}{\nu!} \nonumber \\
&= \sum_{\nu=0}^\infty\left( \sum_{j=0}^{\nu}  C^{(\nu)}_{i,i+j} x_{0}^{[i+j]}\delta_{i\leq N-j}\right)~ \frac{t^\nu}{\nu!}. \label{eq:yihat_analytic}
\end{align}

\subsubsection{Power series of the error}

Using the previous expressions \eqref{eq:yi_analytic} and \eqref{eq:yihat_analytic}, the error in the $i$-th block can be expanded as
\begin{subequations}
\begin{align*}
\eta_i = y_i - \hat{y}_i &= \sum_{\nu=0}^\infty (\chi_{i,\nu} - \hat{\chi}_{i,\nu} ) \frac{t^\nu}{\nu!} \\
&= \sum_{\nu=0}^\infty \left[ \sum_{j=0}^\nu \left( C^{(\nu)}_{i,i+j} - C^{(\nu)}_{i,i+j} \delta_{j\leq N-i} \right)  x_{0}^{[i+j]} \right] \frac{t^\nu}{\nu!} \\
&= \sum_{\nu=N-i+1}^\infty  \sum_{j=N-i+1}^\nu  C^{(\nu)}_{i,i+j}    x_{0}^{[i+j]}  \frac{t^\nu}{\nu!}. \label{eq:errork2dev}
\end{align*}
\end{subequations}
The error in the first block corresponds to setting $i=1$ above,
\begin{equation}
\varepsilon(t) =  \sum_{\nu=N}^\infty \sum_{j=N}^\nu C^{(\nu)}_{1,1+j} x_0^{[1+j]}  \frac{t^\nu}{\nu!} = \sum_{\nu=0}^\infty \xi_{\nu}  \frac{t^\nu}{\nu!}, \label{eq:errork2sumnu}
\end{equation}
where we have defined
\begin{equation}
\xi_\nu :=
\begin{cases} 0 &\text{ if } 0 \leq \nu \leq N-1\\ 
\sum\limits_{j=N}^{\nu} C^{(\nu)}_{1,1+j}x_0^{[1+j]} &\text{ if } \nu \geq  N
\end{cases} \label{eq:errork2xi}
\end{equation}
To conclude the proof, in the remaining of this section we find a condition on the behaviour with $N$ on the coefficients $\xi_\nu$, quantified in terms of their norm $||\xi_\nu||$ for increasing $N$. 

\subsubsection{Solution using generating functions}

In this subsection we are considering the case $k=2$, hence the path order is at most $1$, that is, we set once and for all $q=1$. The objects $C^{(\nu)}_{i,i+j}$ are defined for all $i\geq 1$, $\nu \geq 1$, $0\leq j \leq \nu$. They satisfy the general recurrence formula (here $i\geq 1$, $\nu \geq 2$, $1 \leq j \leq \nu$)
\begin{equation}
C^{(\nu)}_{i,i+j} = C^{(\nu-1)}_{i,i+j-1}C^{(1)}_{i+j-1,i+j} + C^{(\nu-1)}_{i,i+j}C^{(1)}_{i+j,i+j} \delta_{j \leq \nu-1}. \label{eq:recformk2}
\end{equation}
For all $i\geq 1$, $\nu \geq 1$, $0 \leq j \leq \nu$ consider the sequence of norms
\begin{equation*}
c_{i,\nu, j} = ||C^{(\nu)}_{i,i+j}||.
\end{equation*}
The border conditions are $c_{i,1,0} \leq i ||F_1||$, $c_{i,1,1} \leq i ||F_2||$, and for $\nu \geq 1$, $c_{i,\nu,0} \leq (i ||F_1||)^\nu$. Taking the norm on both sides of the recurrence formula \eqref{eq:recformk2} and applying the triangular inequality, we obtain
\begin{equation}
c_{i,\nu,j} \leq c_{i,\nu-1,j-1} (i+j-1) ||F_2|| + c_{i,\nu-1,j} (i+j)||F_1|| \delta_{j \leq \nu-1}. \label{eq:recurrenceThreeIndices}
\end{equation}
The proof of the following key Lemma is presented in Appendix \ref{app:ProofRecurrencePowerSeries}.
\begin{lemma} \label{lemma:recurrenceK2}
The coefficients $c_{i,\nu,j}$ satisfy, for all $i\geq 1$, $\nu \geq 0$, $0 \leq j \leq \nu$,
\begin{equation}
c_{i,\nu,j} \leq ||F_1||^{\nu-j}~||F_2||^j  \binom{i+j-1}{j} \sum_{k=0}^j \binom{j}{k}  (-1)^{j-k} (i+k)^\nu.  \label{eq:RecurrenceResult}
\end{equation}
\end{lemma}

From \eqref{eq:errork2sumnu}-\eqref{eq:errork2xi}, we deduce that $||\varepsilon(t)|| \leq \sum_{\nu=0}^\infty ||\xi_\nu|| \dfrac{t^\nu}{\nu!}$, and $\xi_\nu = 0$ for all $\nu < N$. For $\nu \geq N$, we deduce from Lemma \ref{lemma:recurrenceK2} the estimate
\begin{subequations}
\begin{align*}
||\xi_\nu|| &\leq  \sum_{j=N}^\nu || C^{(\nu)}_{1,1+j} x_0^{[1+j]}|| \\ 
&\leq \sum_{j=N}^\nu ||F_1||^{\nu-j}~||F_2||^j  \sum_{k=0}^j \binom{j}{k}  (-1)^{j-k} (1+k)^\nu ||x_0||^{1+j}. 
\end{align*}
\end{subequations}
Substitution into the error series yields $||\varepsilon(t)|| \leq R_N(t)$, where
\begin{equation}
R_N(t) :=  \sum_{\nu=N}^\infty \sum_{j=N}^\nu ||F_1||^{\nu-j}~||F_2||^j  \sum_{k=0}^j \binom{j}{k}  (-1)^{j-k} (1+k)^\nu ||x_0||^{1+j} \dfrac{t^\nu}{\nu!}. \label{eq:e1i1k2eval}
\end{equation}
This infinite series can be explicitly computed using Egorychev's method for the evaluation of binomial coefficient sums using complex analysis \cite{egorychev1984integral}.

\begin{proof}[Proof of Theorem 2.] Let $\beta_0 = \frac{||x_0||~||F_2||}{||F_1||}$, and let us rewrite \eqref{eq:e1i1k2eval} as
\begin{equation}
R_N(t) = ||x_0|| \sum_{\nu=N}^\infty \frac{(\norm{F_1}t)^\nu}{\nu!}  \sum_{j=N}^\nu \beta_0^j \sum_{k=0}^j \binom{j}{k} (-1)^{j-k} (1+k)^\nu  \label{eq:RNt}
\end{equation}
It is well-known that
\begin{equation*}
(1+k)^\nu =
\frac{\nu!}{2\pi \ii}
\oint_{|z|=\epsilon} \dfrac{e^{(1+k)z}}{z^{\nu+1}}  \dd z.
\end{equation*}
For the inner sum, we get
\begin{equation*}
\sum_{k=0}^j \binom{j}{k} (-1)^{j-k} (1+k)^\nu = \frac{\nu!}{2\pi\ii} \oint_{|z|=\epsilon} \dfrac{e^z}{z^{\nu+1}} (e^z-1)^j \dd z.
\end{equation*}
Performing the sum\footnote{Recall the formula for the shifted sum,
\begin{equation*}
\sum_{j=a}^b x^j =  \frac{x^a-x^{b+1}}{1-x},\qquad 0\leq a \leq b
\end{equation*}} over $j$ in \eqref{eq:RNt}, we get
\begin{equation*}
\frac{\nu!}{2\pi \ii}
\oint_{|z|=\epsilon} \dfrac{e^z}{z^{\nu + 1}} \dfrac{\left[ \beta_0(e^z-1)\right]^{\nu+1} - \left[ \beta_0(e^z-1)\right]^{N} }{\beta_0 e^z-(1+\beta_0)} \dd z.
\end{equation*}
Now since $e^z - 1$ starts at $\nu$, by Cauchy's integral theorem the first term drops out and we get
\begin{equation*}
\frac{\nu!}{2\pi \ii}
\oint_{|z|=\epsilon} \dfrac{e^z}{z^{\nu + 1}} \dfrac{\left[ \beta_0(e^z-1)\right]^{N}}{(1+\beta_0) - \beta_0 e^z} \dd z.
\end{equation*}
We thus get for the remaining sum
\begin{equation*}
\sum_{\nu=N}^\infty \frac{(\norm{F_1}t)^\nu}{\nu!} \frac{\nu!}{2\pi \ii} \oint_{|z|=\epsilon} \dfrac{e^z}{z^{\nu + 1}} \dfrac{\left[ \beta_0(e^z-1)\right]^{N}}{(1+\beta_0) - \beta_0 e^z} \dd z.
\end{equation*}
Finally note  that $(e^z-1)^N$ starts at $z^N$ so (again by Cauchy's integral theorem) for all values $0 \leq \nu \leq N-1$ the integral vanishes, hence we may lower the initial value of the remaining summation to zero without changing its value, getting
\begin{equation*}
\sum_{\nu=0}^\infty (\norm{F_1}t)^\nu \frac{1}{2\pi \ii}
\oint_{|z|=\epsilon} \frac{e^z}{z^{\nu+1}} 
\frac{\left[ \beta_0(e^z-1)\right]^N}{(1+\beta_0) - \beta_0 e^z } \dd z.
\end{equation*}
We deduce that 
\begin{equation*}
R_N(t) =  \dfrac{||x_0|| e^{||F_1||t}}{(1+\beta_0) - \beta_0 e^{||F_1||t}} \left[ \beta_0 (e^{||F_1||t}-1)\right]^N.
\end{equation*}
The distance to the nearest singularity is $||F_1||^{-1}\ln (1+\beta_0^{-1})$, so the radius of convergence of the series is 
\begin{equation*}
|t| < T^\ast := \frac{1}{||F_1||}\ln \left( 1 + \dfrac{1}{\beta_0} \right).
\end{equation*}
Moreover, 
\begin{equation*}
\lim_{N\to \infty} R_N(t) = 0,\qquad \text{~for all~} t < T^\ast.
\end{equation*}

\end{proof}

%% file: conclusion.tex
In this paper we have found explicit error bounds for the solution obtained by truncation at finite orders of the infinite-dimensional Carleman embedding, in the case of polynomial ODEs. We have shown that these error bounds provide a reasonably good estimate in the convergence region, but for practical application of this method, let us raise some questions, which are left for future work:

\begin{itemize}
	\item The error estimate is a time-dependent function computed by expanding a solution around some initial value, hence the accuracy of the error formula depends strongly on the initial condition. The range of validity could be extended, for instance, by space discretization \cite{weber2016adapting}. A different approach would be to discretize in time, thus having a (single) global linearization over a set of timed switches.  
	
	\item We have used monomials basis to perform the Carleman linearization, for ease of notation and theoretical manipulations. However, more accurate finite-dimensional approximations may be obtained by using another set of basis functions, such as Chebyshev polynomials, as already hinted in \cite{bellman1962some}.
	
	\item We have considered the simplifying assumption that zero is an equilibrium point of the nonlinear IVP \eqref{eq:S_ode}. However, the methodology could be extended to handle input functions $u(t) \in \mathcal{U}$, piecewise continuous and valued over a bounded set $\mathcal{U}\subset\R^m$. The Carleman linearization scheme can be constructed accordingly \cite{kowalski1991nonlinear}.
	\end{itemize}

%% file: proofpropintegrals.tex
%

\begin{proof}
Recall that we defined
\begin{align}
\begin{split}
G_N(a, t) &:= \int_{0}^t  e^{a (t-s_{N-1})}   \int_{0}^{s_{N-1}}  e^{2a(s_{N-1} - s_{N-2})} \\
&\cdots  \int_{0}^{s_2}  e^{(N-1)a(s_2 - s_1)}   \int_{0}^{s_1} e^{Na(s_1 - s)} \dd s\dd s_{1}\cdots\dd s_{N-1}.
\end{split}
\end{align}
Observe that $G_N(a,s_N)$ is continuous in $a$ since $$(a,s_0,\ldots,s_N)\mapsto e^{a \left( -Ns_0 +  \sum\limits_{i=1}^{N} s_i \right)}$$
is continuous, integrable and bounded by an integrable function of $(s_0,\ldots,s_N)$ on
every compact set for each $a \in \R$. Thus we can assume that $a\neq0$ in what follows and
conclude by continuity on $0$.

Define $X=\{(s_0,\ldots,s_{N-1}):0\leqslant s_0\leqslant s_1\leqslant\ldots\leqslant s_{N-1}\leqslant s_N\}$ and observe that
\begin{subequations}
\begin{align*}
G_N(s_N) &= \int_{0}^{s_N} \cdots \int_{0}^{s_2} \int_{0}^{s_1} 
    e^{a \left( -Ns_0 +  \sum_{i=1}^{N} s_i \right)} \dd s_0 \dd s_1 \ldots \dd s_{N-1}\\
    &= \int_X \underbrace{e^{a\sum\limits_{i=1}^{N} s_i}}_{:=f(s)}
        \underbrace{\int_{0}^{s_1}e^{-a Ns_0}\dd s_0}_{:=g(s_0)} \dd s\\
    &=\int_Xf(s)g(s_0)\dd s.
\end{align*}
\end{subequations}
For any permutation $\sigma$, let $x_\sigma=\sigma(x)$ and $X_\sigma=\{x_\sigma:x\in X\}$.
Now observe that $f(s)$ is symmetric in $s_1,\ldots,s_{N-1}$, that is $f(s)=f(s_\sigma)$ for any permutation
$\sigma$ of $\{0,\ldots,N-1\}$ that leaves $0$ unchanched. Let $\mathcal{G}$ be the
set of such permutations.
Furthermore, for any two distincts permutations $\sigma\neq\sigma'$, the set $X_\sigma\cap X_{\sigma'}$ has
empty interior. It follows from this that for any function $F$ we have
\begin{equation}\label{eq:integral_perm_sum}
\sum_{\sigma\in\mathcal{G}}\int_{X_\sigma}F(x)\dd x=\int_{\bigcup_{\sigma\in\mathcal{G}}X_\sigma}F(x)\dd x.
\end{equation}
But using symmetry, we also have that
\begin{subequations}
\begin{align}
\sum_{\sigma\in\mathcal{G}}\int_{X_\sigma}f(s)g(s_0)\dd s
    &=\sum_{\sigma\in\mathcal{G}}\int_{X}f(s_\sigma)g(s_{\sigma(0)})\dd s\\
    &=\sum_{\sigma\in\mathcal{G}}\int_{X}f(s)g(s_0)\dd s\\
    &=(N-1)!\int_{X}f(s)g(s_0)\dd s\label{eq:sum_int_perm_fg}.
\end{align}
\end{subequations}
Finally, observe that:
\begin{equation}\label{eq:union_perm_X}
E=\bigcup_{\sigma\in\mathcal{G}}X_\sigma=[0,s_N]^N\cap\big\{(s_0,\ldots,s_{N-1}):s_0\leqslant\min(s_1,\ldots,s_{N-1})\big\}.
\end{equation}
This can be seen by double inclusion: for any $s\in X$ and $\sigma\in\mathcal{G}$
we have, by definition, $s_0\leqslant\min(s_1,\ldots,s_{N-1})$. Since $\sigma(0)=0$ then
$$s_0\leqslant\min(s_{\sigma(1)},\ldots,s_{\sigma(N-1)})=\min(s_1,\ldots,s_{N-1}).$$
Conversely, consider $s$ such that $s_0\leqslant\min(s_1,\ldots,s_{N-1})$. Then
take $\sigma$ such that $s_{\sigma(1)}\leqslant s_{\sigma(2)}\leqslant\ldots\leqslant s_{\sigma(N-1)}$
(just sort the components of $s$) then $s\in X_\sigma$.
Putting together \eqref{eq:integral_perm_sum}, \eqref{eq:sum_int_perm_fg} and \eqref{eq:union_perm_X},
we get that:
\begin{subequations}
\allowdisplaybreaks
\begin{align*}
\int_Xf(s)g(s_0)\dd s
    &=\frac{1}{(N-1)!}\int_E f(s)g(s_0)\dd s\\
    &=\frac{1}{(N-1)!}\int_{0}^{s_N}\int_{s_0}^{s_N}\cdots \int_{s_0}^{s_N}
        f(s)g(s_0)\dd s_1\cdots\dd s_{N-1}\dd s_0\\
    &=\frac{1}{(N-1)!}\int_{0}^{s_N}e^{-a Ns_0}\int_{s_0}^{s_N}\cdots\int_{s_0}^{s_N}
        e^{a\sum_{i=1}^{N}s_i}\dd s_1\cdots\dd s_{N-1}\dd s_0\\
    &=\frac{e^{a s_N}}{(N-1)!}\int_{0}^{s_N}e^{-a Ns_0}\left(\int_{s_0}^{s_N}
        e^{a u}\dd u\right)^{N-1}\dd s_0\\
    &=\frac{e^{a s_N}}{(N-1)!}\int_{0}^{s_N}e^{-a Ns_0}\left(\frac{e^{s_Na}-e^{s_0a}}{a}\right)^{N-1}\dd s_0.
\end{align*}
\end{subequations}
Developing the integrand with the binomial formula and integrating in $s_0$,
\begin{align*}
G_N(s_N) &= \frac{e^{a s_N}}{a^{N-1}}\int_{0}^{s_N}e^{-a Ns_0}\sum_{k=0}^{N-1}\binom{N-1}{k}\frac{(-1)^{N-1-k}e^{ks_Na+(N-1-k)a s_0}}{(N-1)!}\dd s_0\\
    &=\frac{e^{a s_N}}{a^{N-1}}\sum_{k=0}^{N-1}\frac{(-1)^{N-1-k}e^{ka s_N}}{(N-1-k)!k!}\int_{0}^{s_N}e^{-a(k+1)s_0}\dd s_0\\
    &=\frac{e^{a s_N}}{a^{N-1}}\sum_{k=0}^{N-1}\frac{(-1)^{N-1-k}e^{ka s_N}}{(N-1-k)!k!}\left( \frac{1-e^{-a(k+1)s_N}}{a(k+1)}\right).
\end{align*}
Finally, the factored form can be found by expanding the product and reordering,
\begin{align*}
G_N(s_N) &=\frac{1}{a^{N}}\sum_{k=1}^{N}\frac{(-1)^{N-k}e^{ka s_N}}{(N-k)!k!}
-\frac{1}{a^{N}}\sum_{k=1}^{N}\binom{N}{k}\frac{(-1)^{N-k}}{N!}\\
&=\frac{1}{a^{N}}\sum_{k=1}^{N}\frac{(-1)^{N-k}e^{ka s_N}}{(N-k)!k!}
+\frac{(-1)^N}{a^NN!}
-\frac{1}{a^{N}}\underbrace{\sum_{k=0}^{N}\binom{N}{k} \frac{(-1)^{N-k}}{N!}}_{=0}\\
&=\frac{1}{a^{N}}\sum_{k=0}^{N}\frac{(-1)^{N-k}e^{ka s_N}}{(N-k)!k!} = \dfrac{( e^{a s_N} - 1)^N}{N! a^{N}}.
\end{align*}

\end{proof}

%% file: proofrecurrencepowerseries.tex
\begin{proof}
The cases $\nu=0$ and $\nu=1$ trivially verify \eqref{eq:RecurrenceResult}, by a direct application of Proposition \ref{prop:recproofpsk2}. For the general case, let us assume without loss of generality that $\nu \geq 2$. Moreover, since the subindex $i$ stays fixed, we set $c_{\nu,j} \equiv c_{i,\nu,j}$. It is convenient to displace the recurrence formula \eqref{eq:recurrenceThreeIndices} by $1$ in $\nu$, so that
\begin{equation}
c_{\nu+1,j+1} \leq c_{\nu,j} (i+j) \norm{F_2} + c_{\nu,j+1} (i+j+1)\norm{F_1} \delta_{j \leq \nu-1} \label{eq:recurrenceThreeIndicesDisplaced}
\end{equation}
for $i\geq 1$, $\nu \geq 1$, $0 \leq j \leq \nu$.

Consider, for each $j\geq 0$, the generating function 
\begin{align*}
B_j(z) &= \sum_{\nu=j}^\infty z^\nu c_{\nu,j} = z^j c_{j,j} + z^{j+1} c_{j+1,j} + z^{j+2} c_{j+2,j} + \ldots. 
\end{align*}
Here $z$ is a formal (complex) parameter. By definition we set $c_{0,0} = 1$.  For $j=0$, 
\begin{equation*}
B_0(z) = 1 + z c_{1,0} + z^2 c_{2,0} + \ldots \leq \sum_{\nu=0}^\infty z^\nu (i\norm{F_1})^\nu = \frac{1}{1 - iz\norm{F_1}}.
\end{equation*}
Multiplying on both sides by $z^\nu$ and summing from $\nu=j$ to $\infty$, 
\begin{equation*}
\underbrace{\sum_{\nu=j}^\infty z^\nu c_{\nu+1,j+1}}_{\encircle{$1$}} \leq \underbrace{\sum_{\nu=j}^\infty z^\nu c_{\nu,j} (i+j) \norm{F_2}}_{\encircle{$2$}} + \underbrace{\sum_{\nu=j}^\infty z^\nu c_{\nu,j+1} (i+j+1)\norm{F_1} \delta_{j \leq \nu-1}}_{\encircle{$3$}}.
\end{equation*}
Each of these terms can be evaluated explicitly. The left-most term is
\begin{equation*}
	\encircle{$1$} = z^{-1} (z^{j+1} c_{j+1,j+1} + z^{j+2}c_{j+2,j+1} + \ldots ) = z^{-1} B_{j+1}(z).
\end{equation*}
Then $\encircle{$2$} = (i+j) \norm{F_2} B_j(z)$, and finally
\begin{equation*}
\encircle{$3$} = (i+j+1)\norm{F_1} \sum_{\nu=j+1}^\infty z^\nu c_{\nu,j+1} = (i+j+1)\norm{F_1} B_{j+1}(z).
\end{equation*}
Rearranging and multiplying by $z$,
\begin{equation*}
B_{j+1}(z) (1-z(i+j+1)\norm{F_1}) \leq z(i+j) \norm{F_2} B_j(z),
\end{equation*}
and for sufficiently small $z$,
\begin{equation*}
B_{j+1}(z) \leq \frac{z(i+j) \norm{F_2}}{1-z(i+j+1)\norm{F_1}} B_j(z).
\end{equation*}
Consequently, for any $ j \geq 1$,
\begin{align*}
B_j(z) &\leq \prod_{k=0}^{j-1} \frac{z(i+k)\norm{F_2}}{1-z(i+k+1)\norm{F_1}}~ B_0(z) \\
&\leq \frac{1}{1 - iz\norm{F_1}} \prod_{k=0}^{j-1} \frac{z(i+k)\norm{F_2}}{1-z(i+k+1)\norm{F_1}}.
\end{align*}
Since $\left. \dfrac{\partial^{(j)}}{\partial z^j} z^j \right\vert_{z=0} = j!$ for all $j\geq 1$, then
\begin{equation*}
\left. \frac{\partial^{(j)}}{\partial z^j} B_j(z)\right\vert_{z=0} = \left. \frac{\partial^{(j)}}{\partial z^j} \sum_{\nu=j}^\infty  c_{\nu,j} z^\nu\right\vert_{z=0}  = j! c_{j,j}.
\end{equation*}
Moreover for all $\mu \geq j$, 
\begin{equation*}
\left. \frac{\partial^{(\mu)}}{\partial z^\mu} B_j(z)\right\vert_{z=0} = \left. \frac{\partial^{(\mu)}}{\partial z^\mu} \sum_{\nu=j}^\infty  c_{\nu,j} z^\nu\right\vert_{z=0}  = \mu! c_{\mu,j}.
\end{equation*}
Consequently, for all $j \geq 1$ and $\mu \geq j$, 
\begin{align}\label{eq:cmuj_estimate}
c_{\mu,j} &\leq \frac{1}{\mu!} \frac{\partial^{(\mu)}}{\partial z^\mu} \left. \left\{ \frac{1}{1 - iz\norm{F_1}} \prod_{k=0}^{j-1} \frac{z(i+k)\norm{F_2}}{1-z(i+k+1)\norm{F_1}} \right\} \right\vert_{z=0} \nonumber \\
&= \frac{1}{\mu!} \frac{\partial^{(\mu)}}{\partial z^\mu} \left. \left\{ \frac{(i+j-1)!}{(i-1)!}z^j \norm{F_2}^j \prod_{k=0}^{j} \frac{1}{1-z(i+k)\norm{F_1}} \right\} \right\vert_{z=0} \nonumber \\
&= \frac{(i+j-1)!}{(i-1)!}\norm{F_2}^j \frac{1}{\mu!} \frac{\partial^{(\mu)}}{\partial z^\mu} \left. \left\{ z^j \prod_{k=0}^{j} \frac{1}{1-z(i+k)\norm{F_1}} \right\} \right\vert_{x=0}. 
\end{align}
Let $w = z\norm{F_1}$. By partial fraction decomposition,
\begin{align*}
\prod_{k=0}^j \frac{1}{1-w(i+k)} = \sum_{k=0}^j \frac{\alpha_k}{1-w(i+k)}.
\end{align*}
Matching for $k=0,\ldots,j$ on both sides of the equation, the $\alpha_k$ are given by
\begin{align*}
\alpha_k = \prod_{r=0,r\neq k}^{j} \left. \frac{1}{1-w(i+r)} \right\vert_{w=\frac{1}{i+k}} = \prod_{r=0,r\neq k}^{j} \frac{i+k}{k-r}  = \frac{(i+k)^j (-1)^{j-k}}{k!(j-k)!}.
\end{align*}
Finally, 
\begin{align}\label{eq:partial_frac_explicit}
\frac{1}{\mu!} \frac{\partial^{(\mu)}}{\partial z^\mu} \left. \left\{ z^j \prod_{k=0}^{j} \frac{1}{1-z(i+k)\norm{F_1}} \right\} \right\vert_{z=0}  &= [z^\mu] \left\{ z^j \prod_{k=0}^{j} \frac{1}{1-z(i+k)\norm{F_1}} \right\} \nonumber \\ &= [z^{\mu-j}] \left\{  \prod_{k=0}^{j} \frac{1}{1-z(i+k)\norm{F_1}} \right\} \nonumber \\ &= [z^{\mu-j}] \left\{  \sum_{k=0}^j \frac{\alpha_k}{1-z(i+k)\norm{F_1}} \right\} \nonumber \\ &= \sum_{k=0}^j \alpha_k [z^{\mu-j}] \left\{  \frac{1}{1-z(i+k)\norm{F_1}} \right\} \nonumber\\ &= \sum_{k=0}^j \alpha_k (i+k)^{\mu-j}\norm{F_1}^{\mu-j}.
\end{align}
Combining \eqref{eq:cmuj_estimate} with \eqref{eq:partial_frac_explicit}, and after some rearrangements, we arrive at the desired result,
\begin{align*}
c_{\mu,j} &\leq \frac{(i+j-1)!}{(i-1)!}\norm{F_2}^j \sum_{k=0}^j \alpha_k (i+k)^{\mu-j}\norm{F_1}^{\mu-j} \\ 
&= \frac{(i+j-1)!}{(i-1)!}\norm{F_2}^j \sum_{k=0}^j \frac{(i+k)^j (-1)^{j-k}}{k!(j-k)!} (i+k)^{\mu-j}\norm{F_1}^{\mu-j} \\ 
&= \frac{(i+j-1)!}{(i-1)!}\norm{F_1}^{\mu-j}~\norm{F_2}^j \sum_{k=0}^j \frac{(i+k)^\mu (-1)^{j-k}}{k!(j-k)!} \\ 
&= \norm{F_1}^{\mu-j}~\norm{F_2}^j  \binom{i+j-1}{j} \sum_{k=0}^j \binom{j}{k}  (-1)^{j-k} (i+k)^\mu. 
\end{align*}
\end{proof}